\documentclass[a4paper]{amsart}

\usepackage{a4wide,amsfonts,amssymb,amsmath,color,ifthen,todonotes}
\usepackage[english]{babel}
\usepackage[all]{xy}

\allowdisplaybreaks

\newtheorem{lem}{Lemma}[section]
\newtheorem{defi}[lem]{Definition}
\newtheorem{theo}[lem]{Theorem}
\newtheorem{cor}[lem]{Corollary}
\newtheorem{rem}[lem]{Remark}

\newtheorem{genass}[lem]{General Assumption}
\newtheorem{ex}[lem]{Example}

\newcommand{\ol}{\overline}

\def\To{\to}
\def\cpt{\hookrightarrow}
\def\equi{\Leftrightarrow}
\def\qequi{\quad\equi\quad}

\def\z{\mathbb{Z}}
\def\n{\mathbb{N}}
\def\nz{\n_{0}}
\def\reals{\mathbb{R}}
\def\rt{\reals^{3}}
\def\rttt{\reals^{3\times3}}

\def\om{\Omega}
\def\ga{\Gamma}

\def\RM{\boldsymbol{\mathsf{RM}}}
\def\bbS{\mathbb{S}}
\def\bbT{\mathbb{T}}
\def\sfL{\mathsf{L}}
\def\bsfL{\boldsymbol{\mathsf{L}}}
\def\bcalL{\boldsymbol{\mathcal{L}}}
\def\sfH{\mathsf{H}}
\def\sfV{\mathsf{V}}
\def\bsfH{\boldsymbol{\mathsf{H}}}
\def\bcalH{\boldsymbol{\mathcal{H}}}
\def\sfC{\mathsf{C}}
\def\bsfC{\boldsymbol{\mathsf{C}}}
\def\bcalC{\boldsymbol{\mathcal{C}}}
\newcommand{\vHarm}[2]{\boldsymbol{\mathsf{Harm}}{}^{#1}_{\mathsf{#2}}}
\newcommand{\tHarm}[2]{\boldsymbol{\mathcal{H}\mathsf{arm}}{}^{#1}_{\mathsf{#2}}}
\renewcommand{\L}[2]{\sfL{}^{#1}_{#2}}
\newcommand{\vL}[2]{\bsfL{}^{#1}_{#2}}
\newcommand{\tL}[2]{\bcalL{}^{#1}_{#2}}
\renewcommand{\H}[2]{\sfH{}^{#1}_{#2}}
\newcommand{\Hc}[2]{\mathring{\sfH}{}^{#1}_{#2}}
\newcommand{\vH}[2]{\bsfH{}^{#1}_{#2}}
\newcommand{\vHc}[2]{\mathring{\bsfH}{}^{#1}_{#2}}
\newcommand{\tH}[2]{\bcalH{}^{#1}_{#2}}
\newcommand{\tHc}[2]{\mathring{\bcalH}{}^{#1}_{#2}}

\newcommand{\Cc}[2]{\mathring{\sfC}{}^{#1}_{#2}}

\newcommand{\vCc}[2]{\mathring{\bsfC}{}^{#1}_{#2}}

\newcommand{\tCc}[2]{\mathring{\bcalC}{}^{#1}_{#2}}

\DeclareMathOperator{\supp}{supp}
\DeclareMathOperator{\sym}{sym}
\DeclareMathOperator{\skw}{skw}
\DeclareMathOperator{\dev}{dev}
\DeclareMathOperator{\spn}{spn}
\DeclareMathOperator{\tr}{tr}
\DeclareMathOperator{\p}{\partial}
\DeclareMathOperator{\id}{id}
\DeclareMathOperator{\na}{\nabla}
\DeclareMathOperator{\nac}{\mathring{\na}}
\DeclareMathOperator{\devna}{\dev\nabla}
\DeclareMathOperator{\symna}{\sym\nabla}
\DeclareMathOperator{\symnac}{\mathring{\symna}}
\DeclareMathOperator{\rot}{rot}
\DeclareMathOperator{\rotc}{\mathring{\rot}}
\DeclareMathOperator{\Rot}{Rot}
\DeclareMathOperator{\Rotc}{\mathring{\Rot}}

\DeclareMathOperator{\Rotcs}{\Rotc_{\bbS}}
\DeclareMathOperator{\Rott}{\Rot^{\top}}
\DeclareMathOperator{\RotRot}{\Rot\Rot}
\DeclareMathOperator{\RotRotc}{\mathring{\RotRot}}
\DeclareMathOperator{\RotRott}{\RotRot^{\top}}
\DeclareMathOperator{\RotRotst}{\RotRot_{\bbS}^{\top}}
\DeclareMathOperator{\RotRotcst}{\RotRotc_{\bbS}^{\top}}
\DeclareMathOperator{\symRot}{\sym\Rot}
\DeclareMathOperator{\symRott}{\symRot_{\bbT}}
\def\div{\operatorname{div}}
\DeclareMathOperator{\divc}{\mathring{\div}}
\DeclareMathOperator{\Div}{Div}
\DeclareMathOperator{\Divc}{\mathring{\Div}}
\DeclareMathOperator{\Divs}{\Div_{\bbS}}
\DeclareMathOperator{\Divcs}{\Divc_{\bbS}}

\DeclareMathOperator{\Divct}{\Divc_{\bbT}}
\def\vv{\boldsymbol{v}}
\def\vu{\boldsymbol{u}}

\def\vx{\boldsymbol{x}}
\def\vb{\boldsymbol{b}}
\def\vd{\boldsymbol{d}}
\def\vp{\boldsymbol{p}}
\def\vq{\boldsymbol{q}}
\def\vr{\boldsymbol{r}}

\def\vphi{\boldsymbol{\phi}}

\def\tI{\boldsymbol{I}}
\def\tA{\boldsymbol{A}}
\def\tE{\boldsymbol{E}}
\def\tM{\boldsymbol{M}}
\def\tN{\boldsymbol{N}}

\def\tP{\boldsymbol{P}}
\def\tQ{\boldsymbol{Q}}
\def\tPhi{\boldsymbol{\Phi}}

\def\A{\operatorname{A}}
\def\cA{\operatorname{\mathcal{A}}}

\def\Pot{\operatorname{P}}
\def\Qot{\operatorname{Q}}

\DeclareMathOperator{\Az}{\A_{0}}
\DeclareMathOperator{\Ao}{\A_{1}}
\DeclareMathOperator{\At}{\A_{2}}

\DeclareMathOperator{\Azs}{\A_{0}^{*}}
\DeclareMathOperator{\Aos}{\A_{1}^{*}}
\DeclareMathOperator{\Ats}{\A_{2}^{*}}

\DeclareMathOperator{\cAo}{\cA_{1}}

\DeclareMathOperator{\cAzs}{\cA_{0}^{*}}

\newcommand{\norm}[1]{|#1|}

\newcommand{\scp}[2]{\langle#1,#2\rangle}

\title[The Elasticity Complex]
{The Elasticity Complex:\\ Compact Embeddings and Regular Decompositions}
\author{Dirk Pauly}
\address{Fakult\"at f\"ur Mathematik,
Universit\"at Duisburg-Essen, Campus Essen, Germany}
\email[Dirk Pauly]{dirk.pauly@uni-due.de}
\author{Walter Zulehner}
\address{Johann Radon Institute for Computational and Applied Mathematics,
Austrian Academy of Sciences, Altenberger Stra{\ss}e 69, 4040 Linz, Austria}
\email[Walter Zulehner]{zulehner@numa.uni-linz.ac.at}
\date{\today; Corresponding Author: Dirk Pauly}
\thanks{}

\begin{document}

\def\titlerepude{\sf The Elasticity Complex:\\ Compact Embeddings and Regular Decompositions}
\def\authorrepude{Dirk Pauly \& Walter Zulehner}
\def\daterepdue{March 11, 2021}
\def\reportudemathyesno{no}
\def\reportudemathnumber{SM-UDE-824}
\def\reportudemathyear{2021}
\def\reportudematheingang{\daterepdue}
\newcommand{\preprintudemath}[5]{
\thispagestyle{empty}
\begin{center}\normalsize SCHRIFTENREIHE DER FAKULT\"AT F\"UR MATHEMATIK\end{center}
\vspace*{5mm}
\begin{center}#1\end{center}
\vspace*{5mm}
\begin{center}by\end{center}
\vspace*{0mm}
\begin{center}#2\end{center}
\vspace*{5mm}
\normalsize 
\begin{center}#3\hspace{69mm}#4\end{center}
\newpage
\thispagestyle{empty}
\vspace*{210mm}
Received: #5
\newpage
\addtocounter{page}{-2}
\normalsize
}
\ifthenelse{\equal{\reportudemathyesno}{yes}}
{\preprintudemath{\titlerepude}{\authorrepude}{\reportudemathnumber}{\reportudemathyear}{\reportudematheingang}}
{}


\begin{abstract}
We investigate the Hilbert complex of elasticity involving spaces of symmetric tensor fields. 
For the involved tensor fields and operators we show closed ranges, 
Friedrichs/Poincar\'e type estimates, Helmholtz type decompositions, 
regular decompositions, regular potentials, 
finite cohomology groups, and, most importantly, new compact embedding results. 
Our results hold for general bounded strong Lipschitz domains of arbitrary topology
and rely on a general functional analysis framework (FA-ToolBox).
Moreover, we present a simple technique to prove the compact embeddings
based on regular decompositions/potentials and Rellich's selection theorem, 
which can be easily adapted to any Hilbert complex.
\end{abstract}


\vspace*{-5mm}
\maketitle
\setcounter{tocdepth}{4}
\tableofcontents


\section{Introduction}

This paper is devoted to a discussion of the elasticity complex
\begin{equation} 
\label{ElastCompl}
\xymatrixcolsep{9ex}
\xymatrix{
\vL{2}{}(\om) \ar[r]^-{\symna} &
\tL{2}{\bbS}(\om) \ar[r]^-{\RotRotst} &
\tL{2}{\bbS}(\om) \ar[r]^-{\Divs} &
\vL{2}{}(\om),
}
\end{equation}
which plays an important role in the analysis of the elasticity equations as well as in the construction 
and analysis of appropriate numerical methods.

The elasticity complex is a Hilbert complex, i.e., it consists of a sequence of Hilbert spaces, here 
involving $\vL{2}{}(\om)$ and $\tL{2}{\bbS}(\om)$, the spaces of square-integrable vector fields 
and symmetric tensor fields on a domain $\om\subset\reals^3$,
and a sequence of densely defined and closed linear operators, here the (unbounded) differential 
operators $\symna$, $\RotRotst$, $\Divs$, which denote the symmetric part of the gradient 
of vector fields, the Saint-Venant operator, given by
$$\RotRotst\tM=\Rot\big((\Rot\tM)^\top\big)$$
for symmetric tensor fields $\tM$, where $\Rot$ is the (row-wise) rotation of a tensor field, 
and the (row-wise) divergence  of symmetric tensor fields, respectively.

The key property of a complex is that the superposition of two successive operators vanishes, here
$$\RotRotst\,\symna=0,\quad 
\Divs\,\RotRotst=0.$$

In particular, we will consider the operators in \eqref{ElastCompl} on their respective domains of definition 
of fields with homogeneous boundary conditions
$$\vHc{}{}(\symna,\om)=\vHc{1}{}(\om)\subset\vL{2}{}(\om),\quad
\tHc{}{\bbS}(\RotRott,\om)\subset\tL{2}{\bbS}(\om),\quad
\tHc{}{\bbS}(\Div,\om)\subset\tL{2}{\bbS}(\om),$$
which leads to the associated (primal) domain complex
$$\xymatrixcolsep{9ex}
\xymatrix{
\vHc{1}{}(\om)\ar[r]^-{\symnac}&
\tHc{}{\bbS}(\RotRott,\om)\ar[r]^-{\RotRotcst}&
\tHc{}{\bbS}(\Div,\om)\ar[r]^-{\Divcs}&
\vL{2}{}(\om),
}$$
where $\symnac$, $\RotRotcst$, $\Divcs$ indicate the restrictions of $\symna$, $\RotRotst$, $\Divs$ 
to these domains of definition. This domain complex consists of bounded linear operators on their respective domains of definition, 
which are Hilbert spaces if equipped with the respective graph norms.

The corresponding dual domain complex is given by
$$\xymatrixcolsep{9ex}
\xymatrix{
\vL{2}{}(\om) & \ar[l]_-{-\Divs} 
\tH{}{\bbS}(\Div,\om) & \ar[l]_{\RotRotst} 
\tH{}{\bbS}(\RotRott,\om) & \ar[l]_-{-\symna} 
\vH{1}{}(\om)
}$$
with the respective domains of definition of fields with no boundary conditions
$$\tH{}{\bbS}(\Div,\om) \subset \tL{2}{\bbS}(\om), \quad
  \tH{}{\bbS}(\RotRott,\om) \subset \tL{2}{\bbS}(\om), \quad
  \vH{}{}(\symna,\om) = \vH{1}{}(\om) \subset \vL{2}{}(\om) ,$$
see Section \ref{sec:Not}
for detailed definitions of the involved spaces and operators.

As novel and main result of this paper we will show that the embeddings
\begin{align}
\label{intro:cptemb}
\begin{split}
\tHc{}{\bbS}(\RotRott,\om)\cap\tH{}{\bbS}(\Div,\om) & \cpt\tL{2}{\bbS}(\om),\\ 
\tHc{}{\bbS}(\Div,\om)\cap\tH{}{\bbS}(\RotRott,\om) & \cpt\tL{2}{\bbS}(\om)
\end{split}
\end{align}
are compact for bounded strong Lipschitz domains $\Omega$, see Theorem \ref{cptembela1}. 
The proofs are based on the existence of regular potentials and regular decompositions, 
in particular, if $\Omega$ is additionally topologically trivial.
From these compact embeddings, closed ranges of the involved operators, 
finite-dimensional cohomology groups, 
geometric Helmholtz type decompositions, and
Friedrichs/Poincar\'e type estimates
follow then by abstract functional analytic arguments. For details see the FA-ToolBox, e.g., 
from \cite[Section 2.1]{paulyzulehner2019a}
or \cite{paulymaxconst2,P2019b,P2019a,paulyapostfirstordergen},
and, in particular, Theorem \ref{minitbtheo1} and the results 
presented in \cite{2020arXiv200111007P}. 

Similar results are well-known, e.g., for the classical de Rham domain complex in 3D
\begin{equation}
\label{derhamcompl}
\begin{gathered}
\xymatrixcolsep{5ex}\xymatrixrowsep{1ex}
\xymatrix{
\Hc{1}{}(\om) \ar[r]^-{\nac} & \vHc{}{}(\rot,\om) \ar[r]^-{\rotc} & \vHc{}{}(\div,\om) \ar[r]^-{\divc} & \L{2}{}(\om), 
\\
\mathsf{L}^{2}(\om) & 
\ar[l]_-{-\div} \vH{}{}(\div,\om) & 
\ar[l]_-{\rot} \vH{}{}(\rot,\om) & 
\ar[l]_-{-\na} \H{1}{}(\om) 
}
\end{gathered}
\end{equation}
with the gradient $\na$ of scalar functions, the rotation $\rot$ and the divergence $\div$ of vector fields, 
on their respective domains of definition of fields with and without boundary conditions,  
and also for the classical de Rham complex in ND
(or even on Riemannian manifolds) using differential forms.
Note that in the de Rham complex (of any dimension) no second order operators occur, contrary to the elasticity complex. 
Here, the famous Weck selection theorems \cite{weckmax}, i.e., 
the compactness of the embeddings
$$\vHc{}{}(\rot,\om)\cap\vH{}{}(\div,\om)\cpt\vL{2}{}(\om),\qquad
\vHc{}{}(\div,\om)\cap\vH{}{}(\rot,\om)\cpt\vL{2}{}(\om),$$
for bounded (possibly non-smooth) domains play the central role. 
In Weck's work \cite{weckmax} one finds the first result for piece-wise smooth domains,
and in Weber's paper \cite{webercompmax} the first proof being valid for strong Lipschitz domains is given.
See Picard's contribution \cite{picardcomimb} for the first result for weak Lipschitz domains.
For the most recent state of the art results holding for weak Lipschitz domains
and even for mixed boundary conditions see 
\cite{bauerpaulyschomburgmcpweaklip} and 
\cite{bauerpaulyschomburgmaxcompweakliprNarxiv,bauerpaulyschomburgmaxcompweakliprN}.

We will follow the approach taken in \cite{paulyzulehner2019a} 
but with some significant improvements.
In \cite{paulyzulehner2019a} we investigated the 
Gradgrad-domain complex and its dual the divDiv-domain complex, i.e.,
\begin{equation} 
\label{divDivcompl}
\begin{gathered}
\xymatrixcolsep{9ex}\xymatrixrowsep{1ex}
\xymatrix{
\Hc{2}{}(\om) \ar[r]^-{\mathring{\na\na}} &
\tHc{}{\bbS}(\Rot,\om) \ar[r]^-{\Rotcs} &
\tHc{}{\bbT}(\Div,\om) \ar[r]^-{\Divct} &
\vL{2}{}(\om),
\\
\L{2}{}(\om) & 
\ar[l]_-{\div\Divs} \tH{}{\bbS}(\div\Div,\om) &
\ar[l]_-{\symRott} \tH{}{\bbT}(\symRot,\om) &
\ar[l]_-{-\devna} \vH{1}{}(\om), 
}
\end{gathered}
\end{equation}
where the subscript $\bbT$ refers to spaces and operators of deviatoric tensor fields, $\na\na$ is the Hessian of scalar functions.
These complexes are related to the biharmonic equation and the Einstein-Bianchi equations from general relativity. 
We have shown that the embeddings\footnote{A closer inspection of the respective proof of \cite[Lemma 3.22]{paulyzulehner2019a} 
shows immediately that even the embedding
$$\tHc{}{\bbS}(\Rot,\om)\cap\tH{0,-1}{\bbS}(\div\Div,\om)\cpt\tL{2}{\bbS}(\om)$$ 
is compact.}
\begin{align*}
\tHc{}{\bbS}(\Rot,\om)\cap\tH{}{\bbS}(\div\Div,\om)\cpt\tL{2}{\bbS}(\om),\\
\tHc{}{\bbT}(\Div,\om)\cap\tH{}{\bbT}(\symRot,\om)\cpt\tL{2}{\bbT}(\om)
\end{align*}
are compact, see \cite[Lemma 3.22]{paulyzulehner2019a}.

We emphasize that all these results fit into the common framework of Hilbert domain complexes of the form
\begin{equation}
\label{hilcomplexgen}
\begin{gathered}
\xymatrixcolsep{4ex}\xymatrixrowsep{1ex}
\xymatrix{
D(\Az) \ar[r]^-{\Az} & 
D(\Ao) \ar[r]^-{\Ao} & 
\H{}{2}, \\
\H{}{0} & \ar[l]_-{\Azs}
D(\Azs) & \ar[l]_-{\Aos}  
D(\Aos) & 
}
\end{gathered}
\end{equation}
and compact embeddings 
$$D(\Ao)\cap D(\Azs)\cpt\H{}{1}$$
for densely defined and closed (unbounded) linear operators
$$\Az:D(\Az)\subset\H{}{0}\to\H{}{1},\qquad
\Ao:D(\Ao)\subset\H{}{1}\to\H{}{2}$$
on three Hilbert spaces $\H{}{0}$, $\H{}{1}$, $\H{}{2}$ with (Hilbert space) adjoints
$$\Azs:D(\Azs)\subset\H{}{1}\to\H{}{0},\qquad
\Aos:D(\Aos)\subset\H{}{2}\to\H{}{1}.$$
The key property of a complex reads generally $\Ao\Az=0$, i.e.,
$$R(\Az)\subset N(\Ao)$$ 
or equivalently $R(\Aos)\subset N(\Azs)$.
Here and throughout this paper,
the domain of definition, the kernel, and the range of a linear operator $\A$
are denoted by $D(\A)$, $N(\A)$, and $R(\A)$, respectively.

Some important applications, such as proper solution theories 
for first or second order systems, 
various variational formulations, and a posteriori error estimates
for the respective operators of a Hilbert complex
are discussed in detail in \cite{paulyapostfirstordergen}.

The elasticity complex is subject to past, recent, and ongoing research.
Among others, we cite the important works in 
\cite{zbMATH07045589},
\cite{zbMATH05763606},
\cite{arnoldfalkwintherfemec,zbMATH05130996,zbMATH05190885,zbMATH05696861},
\cite{zbMATH01837365},
\cite{zbMATH01368677,zbMATH01551744}
\cite{zbMATH05045995,zbMATH05129883},
\cite{ciarletciarletgeymonatkrasucki2007},
\cite{zbMATH05013398,zbMATH05492696}
as well as the vast literature cited therein. 
When writing this manuscript, cf. our arXiv-preprint \cite{paulyzulehner2020aarxiv}, we were 
not aware of papers on the elasticity complex dealing with compact embeddings.
After completion we found the recent arXiv-preprint \cite{arnoldhu2020a}
on which the following remark is about.

\begin{rem}
\label{rem:BGG}
In this paper as well as in \cite{paulyzulehner2019a} properties of the elasticity complex, 
the $\na\na$-complex and the $\div\Divs$-complex are derived from known properties 
of the de Rham complex based on the same principle. 
The strategy in our work is to show in a first step compact embeddings. 
Then all the other important properties like closed ranges 
and finite-dimensional cohomology groups follow by abstract functional analytic arguments. 
A different approach is taken in \cite{arnoldhu2020a}, 
which deals in a more general way with the question how to derive complexes 
and their properties from known complexes. 
The basic technique in \cite{arnoldhu2020a} is to utilize 
the  Bernstein-Gelfand-Gelfand (BGG) resolution, 
which can be seen as a primarily algebraic procedure for deriving new complexes. 
The first step in \cite{arnoldhu2020a} for analyzing 
these new complexes is to verify that the cohomology groups of the new complexes 
are finite-dimensional provided the cohomology groups to the known complexes are finite dimensional. 
Then all the other important properties like closed ranges and compact embeddings are derived.

In the end the final results on properties like compact embeddings, 
closed ranges, and finite-dimensional cohomology groups, of the two approaches 
- the more algebraic approach in  \cite{arnoldhu2020a} and our more analytic approach - 
coincide to a large extend if applied 
to the elasticity complex, the $\na\na$-complex and the $\div\Divs$-complex.  

In more technical details, the approach in this paper relies on the existence of regular potentials for the de Rham complex for topologically trivial domains, while the approach in \cite{arnoldhu2020a}, if applied to the elasticity complex, is based on the existence of cohomology spaces of smooth functions whose dimensions are independent of the degree of regularity of the involved Sobolev spaces. All these properties are available for the de Rham complex of  Sobolev spaces without boundary conditions and with full Dirichlet boundary conditions, the situation considered in \cite{{arnoldhu2020a}} as well as here.    
The situation for other boundary conditions is not completely clear to us and might differentiate the two approaches.
In any case, we feel strongly that both approaches are very important 
and shed light on new complexes from two different angles.
\end{rem}

The paper is organised as follows. We start in Section \ref{sec:Not} by introducing the notations used in this paper. 
In Subsection \ref{ssec:SpaceDecomp} the concept of potential operators, 
which are the most important tool for the whole paper, is discussed on an abstract level. 
In Subsection \ref{ssec:cptembregdeco} we present a new, elegant, abstract, and short way
to prove compact embedding results for an arbitrary Hilbert complex.
In Section \ref{ssec:DeRham} known results on potential operators for the de Rham complex are recalled.
This will point out the way to prove the corresponding results for the elasticity complex. 
Section \ref{sec:ElastCompl} contains the main results of the paper, 
the existence of potential operators and compact embeddings for the elasticity complex,
and, as a consequence, related space decompositions. 
Readers who are interested in more technical details 
and additional and more sophisticated results
are referred to \cite{2020arXiv200111007P},
which also contains an extended list of results which directly 
follow from the FA-ToolBox. Finally, useful formulas are collected in Appendix \ref{app:Formulas}.

\section{Notations and Preliminaries}
\label{sec:Not}

Throughout the paper lower-case standard letters are used for denoting scalars and scalar functions, 
lower-case boldface letters for vectors and vector fields, and upper-case boldface letters for matrices/tensors and tensor fields.

We will use the following notations from linear algebra for $\tA \in \reals^{3\times3}$, $\vv \in \reals^3$:
$$  \sym \tA := \frac{1}{2}(\tA + \tA^\top) , \quad
  \skw := \frac{1}{2}(\tA - \tA^\top), \quad
  \tr \tA := \sum_{i=1}^3 A_{ii}, \quad
  \spn \vv 
  := \begin{bmatrix} 0 & - v_3 & v_2 \\ v_3 & 0 & -v_1 \\ -v_2 & v_1 & 0 \end{bmatrix}.$$
The identity matrix is denoted by $\tI$ and $\bbS$ is the set of symmetric matrices in $\reals^{3\times3}$.

Moreover, $\na \varphi$ is used for the gradient of a scalar function $\varphi$,
and $\na \vphi$, $\rot \vphi$, $\div \vphi$ denote the Jacobian, rotation, 
divergence of a vector field $\vphi$, respectively. 
For a tensor field $\tPhi$, $\Rot \tPhi$ and $\Div \tPhi$ denote the tensor field 
and the vector field that result from a row-wise application of $\rot$  and $\div$, respectively. 
Additionally, for smooth fields
we introduce the differential operators $\symna$, $\Rott$, $ \RotRott$ by
$$  (\symna) \vphi := \sym (\na \vphi), \quad
  \Rott \tPhi := (\Rot \tPhi)^\top, \quad
  (\RotRott) \tPhi := \Rot (\Rott \tPhi).$$
Depending on the context these differential operators are understood in the strong sense or in the distributional sense.

\subsection{Domains}
\label{sec:domains}

Throughout this paper we suppose the following:

\begin{genass}
\label{genass1}
$\om\subset\rt$ is a bounded strong
Lipschitz domain with boundary $\ga$.
\end{genass}

This means that $\ga$ is locally a graph of some Lipschitz function.

\begin{defi}[Trivial topology]
A bounded strong Lipschitz domain $\om$ with boundary $\Gamma$ is called additionally topologically trivial, 
if $\Omega$ 
has no handles (also known as $\Omega$ being simply connected) and $\Gamma$ is connected, i.e.,
all Betti numbers are trivial.
\end{defi}

To use standard localisation techniques, we introduce an open covering $(U_{\ell})$ of $\ol{\om}$,
such that all $\om_{\ell}:=\om\cap U_{\ell}$ are topologically trivial (bounded) strong Lipschitz domains.
As $\ol{\om}$ is compact, there exists a finite subcovering denoted by $(U_{1},\dots,U_{L})$. 
We fix a partition of unity $(\varphi_{\ell})_{\ell=1,\dots,L}$ subordinate to $(U_{\ell})_{\ell=1,\dots,L}$ 
with $\varphi_{\ell}\in\Cc{\infty}{}(U_{\ell})$, $\ell=1,\dots,L$,
and, additionally, $\phi_{\ell}\in\Cc{\infty}{}(U_{\ell})$ 
with $\phi_{\ell}|_{\supp\varphi_{\ell}}=1$, $\ell=1,\dots,L$.

\subsection{FA-ToolBox}
\label{sec:FAToolBox}

We make use of the FA-ToolBox presented in detail in, e.g., 
\cite{paulyzulehner2019a} and \cite{paulymaxconst2,P2019b,P2019a,paulyapostfirstordergen},
together with an extension suited for (bounded linear) regular potential operators
and regular decompositions in Subsection \ref{ssec:SpaceDecomp}.
In Subsection \ref{ssec:cptembregdeco} we present a new, elegant, abstract, and short way
to prove compact embedding results for an arbitrary Hilbert complex. This abstract result will be of fundamental importance for the analysis of the elasticity complex in Section \ref{sec:ElastCompl}.

\subsubsection{Functional Analytical Setting}
\label{sec:FASetting}

The inner product and the norm in a Hilbert space $\H{}{}$ are denoted by $\norm{\cdot}_{\H{}{}}$ 
and $\scp{\cdot}{\cdot}_{\H{}{}}$, respectively. 
We use $\sfV^{\bot_{\H{}{}}}$ for the orthogonal complement of a subspace $\sfV$ of $\H{}{}$. 
The notations $\sfV_{1}+\sfV_{2}$, $\sfV_{1}\dotplus\sfV_{2}$, $\sfV_{1}\oplus_{\H{}{}}\sfV_{2}$ 
indicate the sum, the direct sum, the orthogonal sum of two subspaces $\sfV_{1}$, $\sfV_{2}$ of $\H{}{}$, respectively.
Throughout this paper, the domain of definition, the kernel, and the range of a linear operator $\A$
is denoted by $D(\A)$, $N(\A)$, and $R(\A)$, respectively.

Let $\H{}{0}$, $\H{}{1}$, $\H{}{2}$ be Hilbert spaces.
We shall consider densely defined, closed, and \emph{unbounded} linear operators
$$\A:D(\A)\subset\H{}{1}\to\H{}{2},$$
where unboundedness refers to the norms in $\H{}{1}$ and $\H{}{2}$,
as well as \emph{bounded} linear operators
$$\A:D(\A)\to\H{}{2},$$
where boundedness refers to the norms in $D(\A)$ and $\H{}{2}$ 
with $D(\A)$ as Hilbert space, equipped with the graph inner product. 
The different notations indicate the difference in the boundedness.
Note that, in contrast to the Banach space adjoint of the bounded operator $\A$,
the \emph{Hilbert space adjoint} of the unbounded operator $\A$ is given by
the densely defined, closed, and unbounded linear operator
$$\A^{*}:D(\A^{*})\subset\H{}{2}\to\H{}{1}$$
defined by 
$D(\A^{*})=\{y\in\H{}{2}:\exists\,f\in\H{}{1}\;\forall\,x\in D(\A)\;\;\scp{\A x}{y}_{\H{}{2}}=\scp{x}{f}_{\H{}{1}}\}$
and $\A^{*}y=f$ for $y\in D(\A^{*})$,
and hence characterised by
$$\forall\,y\in D(\A^{*})\quad
\forall\,x\in D(\A)\qquad
\scp{\A x}{y}_{\H{}{2}}=\scp{x}{\A^{*}y}_{\H{}{1}}.$$

\begin{rem}[Geometric Helmholtz type decompositions and reduced operators]
\label{rem:helmtypedeoc1}
We recall that the unbounded operator $\A$ induces
by $\ol{R(\A^{*})}=N(\A)^{\bot_{\H{}{1}}}$
and $\ol{R(\A)}=N(\A^{*})^{\bot_{\H{}{2}}}$ 
the Helmholtz type decompositions
$$\H{}{1}=\ol{R(\A^{*})}\oplus_{\H{}{1}}N(\A),\qquad
\H{}{2}=\ol{R(\A)}\oplus_{\H{}{2}}N(\A^{*}),$$
leading directly to the so-called reduced operators
$\cA:=\A|_{N(\A)^{\bot_{\H{}{1}}}}$ and $\cA^{*}:=\A^{*}|_{N(\A^{*})^{\bot_{\H{}{2}}}}$, i.e.,
\begin{align*}
\cA:D(\cA)\subset\ol{R(\A^{*})}&\to\ol{R(\A)},
&
D(\cA)&:=D(\A)\cap\ol{R(\A^{*})}
=D(\A)\cap N(\A)^{\bot_{\H{}{1}}},\\
\cA^{*}:D(\cA^{*})\subset\ol{R(\A)}&\to\ol{R(\A^{*})},
&
D(\cA^{*})&:=D(\A^{*})\cap\ol{R(\A)}
=D(\A^{*})\cap N(\A^{*})^{\bot_{\H{}{2}}}.
\end{align*}
Then we have also the Helmholtz type decompositions
\begin{align*}
D(\A)&=D(\cA)\oplus_{\H{}{1}}N(\A),
&
D(\A^{*})&=D(\cA^{*})\oplus_{\H{}{2}}N(\A^{*})
\end{align*}
and it holds $R(\A)=R(\cA)$ and $R(\A^{*})=R(\cA^{*})$.
Note that, as $\cA$ and $\cA^{*}$ are injective and by the closed graph theorem,
$\cA^{-1}$ and $(\cA^{*})^{-1}$ are bounded if and only if the ranges $R(\A)$ and $R(\A^{*})$ are closed
if and only if the Friedrichs/Poincar\'e type estimates 
\begin{align*}
\exists\,c&>0\quad\forall\,x\in D(\cA)
&\norm{x}_{\H{}{1}}&\leq c\,\norm{\A x}_{\H{}{2}},\\
\exists\,c&>0\quad\forall\,y\in D(\cA^{*})
&\norm{y}_{\H{}{2}}&\leq c\,\norm{\A^{*}y}_{\H{}{1}}
\end{align*}
hold, respectively. Moreover, by the closed range theorem,
$R(\A)$ is closed if and only if $R(\A^{*})$ is closed.
\end{rem}

\subsubsection{Hilbert Complexes and Compact Embeddings}
\label{sec:HilbertComComptEmb}

Our main interest are \emph{Hilbert complexes}
\begin{equation*}
\xymatrixcolsep{4ex}\xymatrixrowsep{1ex}
\xymatrix{
\H{}{0}\ar[r]^-{\Az}& 
\H{}{1}\ar[r]^-{\Ao}& 
\H{}{2},\\
\H{}{0}&\ar[l]_-{\Azs}
\H{}{1}&\ar[l]_-{\Aos}  
\H{}{2}& 
}
\end{equation*}
of densely defined, closed, and unbounded linear operators
$$\Az:D(\Az)\subset\H{}{0}\to\H{}{1},\qquad
\Ao:D(\Ao)\subset\H{}{1}\to\H{}{2}$$
with (Hilbert space) adjoints
$$\Azs:D(\Azs)\subset\H{}{1}\to\H{}{0},\qquad
\Aos:D(\Aos)\subset\H{}{2}\to\H{}{1},$$
and corresponding \emph{compact embeddings}
$$D(\Ao)\cap D(\Azs)\cpt\H{}{1}.$$
Equivalently, we shall denote these Hilbert complexes 
as \emph{Hilbert domain complexes} of bounded linear operators and write
\begin{equation*}
\xymatrixcolsep{4ex}\xymatrixrowsep{1ex}
\xymatrix{
D(\Az)\ar[r]^-{\Az}& 
D(\Ao)\ar[r]^-{\Ao}& 
\H{}{2},\\
\H{}{0}&\ar[l]_-{\Azs}
D(\Azs)&\ar[l]_-{\Aos}  
D(\Aos),& 
}
\end{equation*}
cf.~\eqref{hilcomplexgen}.
The \emph{key property} of a complex reads generally 
\begin{align}
\label{keyprop}
\Ao\Az\subset0,
\end{align}
i.e., $R(\Az)\subset N(\Ao)$ or equivalently $R(\Aos)\subset N(\Azs)$.
Throughout the rest of Section \ref{sec:Not} we assume the following to hold.

\begin{genass}
\label{genass2}
$\xymatrix{\H{}{0}\ar[r]^-{\Az}&\H{}{1}\ar[r]^-{\Ao}&\H{}{2}}$
is a Hilbert complex of densely defined and closed (unbounded) linear operators 
$\Az:D(\Az)\subset\H{}{0}\to\H{}{1}$ and
$\Ao:D(\Ao)\subset\H{}{1}\to\H{}{2}$.
Note that the complex property \eqref{keyprop} holds, i.e.,
$\Ao\Az\subset0$ and $\Azs\Aos\subset0$.
\end{genass}

\subsubsection{Potentials and Decompositions}
\label{ssec:SpaceDecomp}

Let 
$$\A:D(\A)\subset\H{}{1}\to\H{}{2}$$
be a densely defined and closed (unbounded) linear operator.

\begin{defi}[Potential operator]
\label{def:potdef}
Let $\H{+}{1}\subset D(\A)$ be another Hilbert space.
A bounded linear operator 
\begin{equation} 
\label{Pot}
\Pot_{\A}:R(\A)\To\H{+}{1}
\qquad\text{with}\qquad 
\A\Pot_{\A}=\id_{R(\A)}
\end{equation}
is called a potential operator associated to $\A$. 
In short, a potential operator is a bounded right inverse of $\A$. 
\end{defi}

By definition we have, for each $y\in R(\A)$,
\begin{equation} 
\label{zy}
y=\A p\quad\text{with}\quad p:=\Pot_{\A}y.
\end{equation}
We call $p$ a \emph{potential} of $y$.
For each $x \in D(\A)$, we have the representation
\begin{equation} 
\label{stableDecom1}
  x = p + x_0
  \quad \text{with} \quad
  p :=  \Pot_{\A} \A x \in R(\Pot_{\A}) \subset \H{+}{1},\quad
    \A p = \A x,
  \quad \text{and} \quad
  x_0 := x - p \in N(\A),
\end{equation}
which leads to the following space decomposition.

\begin{lem}[Potential decomposition]
\label{lem:SpaceRegDeco}
Let $\Pot_{\A}$ be a potential operator. Then the potential decomposition
$D(\A)= R(\Pot_{\A})\dotplus N(\A)=\H{+}{1}+N(\A)$
and $R(\A)=\A R(\Pot_{\A})=\A\H{+}{1}$ hold.
\end{lem}

Note that the decomposition \eqref{stableDecom1} of Lemma \ref{lem:SpaceRegDeco} 
is stable (topological/continuous/bounded), i.e.,
$$\norm{p}_{\H{+}{1}}\le c^{+}\norm{\A x}_{\H{}{2}},\qquad
\norm{x_{0}}_{\H{}{1}}\le\norm{x}_{\H{}{1}}+c_{1}c^{+}\norm{\A x}_{\H{}{2}},$$
where $c^{+}$ is the norm of $\Pot_{\A}$ and 
$c_{1}$ is the norm of the continuous embedding $\H{+}{1}\subset\H{}{1}$,
which we assume to hold.

\begin{ex}[Geometric potential operator]
\label{ex:geopotop}
Let $R(\A)$ be closed.
The reduced operator of $\A$, cf. Remark \ref{rem:helmtypedeoc1},
$$\cA:D(\cA)\subset R(\A^{*})\To R(\A)$$
is linear and bijective. Its inverse $\Pot_{\A}:=\cA^{-1}:R(\A)\to D(\cA)$ is bounded
(closed graph theorem) and a right inverse of $\A$. 
We call this potential operator the geometric potential operator associated to $\A$. 
It is easy to see that the associated space decomposition is an orthogonal decomposition
$$  D(\A) = D(\cA) \oplus_{\H{}{1}} N(\A), $$
which we call a geometric (or Helmholtz type) decomposition.
Note that the reduced bounded variant
$\cA:D(\cA)\To R(\A)$ is a topological isomorphism.
\end{ex}

The essential tool used in this paper is the existence of potential operators, 
where $\H{+}{1}$ is a smoother space than $D(\cA)$
allowing for a compact embedding $\H{+}{1}\cpt\H{}{1}$,
realised by, e.g., the well known Rellich's selection theorem. 
In this case we speak of \emph{regular potentials} and introduce 
the concept of \emph{regular potential operators}:

\begin{defi}
A potential operator $\Pot_{\A}\colon R(\A)\To\H{+}{1}$ 
with a compact embedding $\H{+}{1}\cpt\H{}{1}$ is called a regular potential operator. 
\end{defi}

Recall our General Assumption \ref{genass2}
and the corresponding Hilbert domain complex 
$$\xymatrix{D(\Az)\ar[r]^-{\Az}&D(\Ao)\ar[r]^-{\Ao}&\H{}{2}}.$$
Let
\begin{equation} 
\label{Potcomplex}
  \Pot_{\Az} \colon R(\Az) \To \H{+}{0}, \qquad
  \Pot_{\Ao} \colon R(\Ao) \To \H{+}{1}
\end{equation}
be associated potential operators. 
Moreover, let us assume that the Hilbert complex is \emph{exact}, i.e.,
\begin{equation} 
\label{exact}
R(\Az) = N(\Ao).
\end{equation}
Using \eqref{stableDecom1} for $\Ao$ and \eqref{zy} for $\Az$ together with 
the exact complex property \eqref{exact}, 
we obtain, for each $x \in D(\Ao)$, the representation 
\begin{equation} 
\label{stableDecom2}
  x = p_{1} + \Az p_{0} 
  \quad \text{with} \quad 
  p_{1} := \Pot_{\Ao} \Ao x,\quad 
  \Ao p_{1} = \Ao x,\quad 
  p_{0} := \Pot_{\Az} (x - p_{1}),
\end{equation}
which leads to a refined version of the space decomposition from above.

\begin{lem}[Refined potential decomposition]
\label{lem:SpaceRegDecorefined}
Let $R(\Az)=N(\Ao)$, 
and let $\Pot_{\Az}$, $\Pot_{\Ao}$ be potential operators. Then
$$D(\Ao)
=R(\Pot_{\Ao})\dotplus\Az R(\Pot_{\Az})
=R(\Pot_{\Ao})\dotplus\Az\H{+}{0}
=\H{+}{1}+\Az\H{+}{0}$$
as well as $R(\Az)=\Az R(\Pot_{\Az})=\Az\H{+}{0}$ and $R(\Ao)=\Ao R(\Pot_{\Ao})=\Ao\H{+}{1}$.
\end{lem}

Decomposition \eqref{stableDecom2} can be rephrased as follows. For each $x \in D(\Ao)$ we have
$$x = \Qot_{\Ao} x + \Az \widetilde \Qot_{\Az} x
\quad\text{with}\quad
\Qot_{\Ao}=\Pot_{\Ao} \Ao
\quad\text{and}\quad
\widetilde\Qot_{\Az}=\Pot_{\Az} (\id_{D(\Ao)} - \Qot_{\Ao}).$$
Note that $\Qot_{\Ao}$ and $\id_{D(\Ao)}-\Qot_{\Ao}$ are projections
and hence $\Qot_{\Ao}$ and $\widetilde\Qot_{\Az}$ are bounded.
More precisely, $\id_{D(\Ao)} - \Qot_{\Ao}$ is a projection onto $N(\Ao)=R(\Az)$. 
So, the operator $\widetilde\Qot_{\Az}$ primarily depends on $\Az$.
The dependency of $\widetilde\Qot_{\Az}$ on $\Ao$ is notationally suppressed for simplicity. 
The operators $\widetilde\Qot_{\Az}$ and $\Qot_{\Ao}$ are bounded, since
\begin{align*}
\norm{\Qot_{\Ao} x}_{\H{+}{1}}&=\norm{p_{1}}_{\H{+}{1}}
\le c_1^+\norm{\Ao x}_{\H{}{2}},\\
\norm{\widetilde\Qot_{\Az} x}_{\H{+}{0}}&=\norm{p_{0}}_{\H{+}{0}}
\le c_{0}^{+}\norm{x-p_{1}}_{\H{}{1}}
\le c_{0}^{+}\norm{x}_{\H{}{1}}
+c_{1}c_{0}^{+}c_{1}^{+}\norm{\Ao x}_{\H{}{2}},
\end{align*}
where $c_{l}^{+}$ is the norm of $\Pot_{\A_{l}}$, $l=0,1$. 

Motivated by this, the concept of a \emph{regular space decomposition} is introduced as follows 
(without any reference to the exactness property \eqref{exact}).

\begin{defi}[Regular decomposition]
\label{def:decompdef}
A decomposition $D(\Ao)=\H{+}{1}+\Az\H{+}{0}$ with compact embeddings $\H{+}{0}\cpt\H{}{0}$ and $\H{+}{1}\cpt\H{}{1}$ is called a regular decomposition, if there are bounded operators 
\begin{equation}
\label{Qotcomplex}
  \Qot_{\Ao} \colon D(\Ao) \To \H{+}{1}, \qquad
  \widetilde\Qot_{\Az} \colon D(\Ao) \To \H{+}{0},
\end{equation}
such that
\begin{equation}
\label{regdecomp}
x = \Qot_{\Ao} x + \Az \widetilde \Qot_{\Az} x
\end{equation}
holds for all $x \in D(\Ao)$. The operators in \eqref{Qotcomplex} are called decomposition operators.
\end{defi}

\subsubsection{Compact Embeddings by Regular Decompositions}
\label{ssec:cptembregdeco}
  
Regular decompositions are very powerful tools.
In particular, compact embeddings can easily be proved,
which then in combination with the FA-ToolBox,
cf.~\cite{paulymaxconst2,P2019b,P2019a,paulyapostfirstordergen,paulyzulehner2019a},
immediately lead to a comprehensive list of important results for the underlying Hilbert complex.

\begin{theo}[Compact embedding by regular decompositions]
\label{cptembmaintheo}
Let $D(\Ao)=\H{+}{1}+\Az\H{+}{0}$ be a regular decomposition.
Then $D(\Ao)\cap D(\Azs)\cpt\H{}{1}$ is compact.
\end{theo}

\begin{proof}
Let $(x_{n})\subset D(\Ao)\cap D(\Azs)$ be a bounded sequence, 
i.e., there exists $c>0$ such that for all $n$ we have
$\norm{x_{n}}_{\H{}{1}}
+\norm{\Ao x_{n}}_{\H{}{2}}
+\norm{\Azs x_{n}}_{\H{}{0}}
\leq c$.
By assumption we decompose 
$x_{n}=p_{1,n}+\Az p_{0,n}$ with $p_{1,n}\in\H{+}{1}$
and $p_{0,n}\in\H{+}{0}$ satisfying
$\norm{p_{1,n}}_{\H{+}{1}}
+\norm{p_{0,n}}_{\H{+}{0}}
\leq c\norm{x_{n}}_{D(\Ao)}
\leq c.$
Hence $(p_{\ell,n})\subset\H{+}{\ell}$ is bounded in $\H{+}{\ell}$, $\ell=0,1$,
and thus we can extract convergent subsequences, again denoted by $(p_{\ell,n})$,
such that $(p_{\ell,n})$ are convergent in $\H{}{\ell}$, $\ell=0,1$. 
Then with $x_{n,m}:=x_{n}-x_{m}$ and $p_{\ell,n,m}:=p_{\ell,n}-p_{\ell,m}$ we get
\begin{align*}
\norm{x_{n,m}}_{\H{}{1}}^2
=\scp{x_{n,m}}{p_{1,n,m}+\Az p_{0,n,m}}_{\H{}{1}}
&=\scp{x_{n,m}}{p_{1,n,m}}_{\H{}{1}}
+\scp{\Azs x_{n,m}}{p_{0,n,m}}_{\H{}{0}}\\
&\leq c\big(\norm{p_{1,n,m}}_{\H{}{1}}
+\norm{p_{0,n,m}}_{\H{}{0}}\big),
\end{align*}
which shows that $(x_{n})$ is a Cauchy sequence in $\H{}{1}$.
\end{proof}  

There is a dual version of Theorem \ref{cptembmaintheo}.

\begin{cor}[Compact embedding by regular decompositions]
\label{cptembmaintheodual}
Let $D(\Azs)=\H{+}{1}+\Aos\H{+}{2}$ be a regular decomposition.
Then $D(\Ao)\cap D(\Azs)\cpt\H{}{1}$ is compact.
\end{cor}

\begin{proof}
The proof follows by duality or it can be done
analogously to the proof of Theorem \ref{cptembmaintheo} with
\begin{align*}
\norm{x_{n,m}}_{\H{}{1}}^2
=\scp{x_{n,m}}{p_{1,n,m}+\Aos p_{2,n,m}}_{\H{}{1}}
&=\scp{x_{n,m}}{p_{1,n,m}}_{\H{}{1}}
+\scp{\Ao x_{n,m}}{p_{2,n,m}}_{\H{}{2}}\\
&\leq c\big(\norm{p_{1,n,m}}_{\H{}{1}}
+\norm{p_{2,n,m}}_{\H{}{2}}\big),
\end{align*}
as crucial estimate.
\end{proof}  

As we have seen, regular potential operators lead quite naturally to regular decompositions if the Hilbert complex is exact. 
Conversely and quite generally, regular decompositions can always be used to construct regular potential operators.

\begin{cor}
\label{regdecompregpot}
Let $D(\Ao)=\H{+}{1}+\Az\H{+}{0}$ be a regular decomposition.
Then the operator
\[
  \Pot_{\Ao} = \Qot_{\Ao} \cA_{1}^{-1}
  \colon R(\Ao) \To \H{+}{1}
\]
is well-defined and a regular potential operator associated to $ \Ao$.
\end{cor}

\begin{proof}
From Theorem \ref{cptembmaintheo} and 
the FA-ToolBox, e.g., \cite[Theorem 2.9]{paulyzulehner2019a}, 
it follows that $R(\Ao)$ is closed. 
Therefore, the reduced operator $\cAo$ of $\Ao$ is linear and bijective with a bounded inverse, 
which makes the operator $\Pot_{\Ao}$ well-defined. 
Let $x\in D(\Ao)$. Applying $\Ao$ to the regular decomposition
$$x=\Qot_{\Ao}x+\Az\widetilde\Qot_{\Az}x$$
yields $\Ao x=\Ao\Qot_{\Ao}x$. Therefore,
$\Ao\Pot_{\Ao}=\Ao\Qot_{\Ao}\cA_{1}^{-1}=\Ao\cA_{1}^{-1}=\id_{R(\Ao)}$.
\end{proof}

Corollary \ref{regdecompregpot} admits a dual version as well, 
cf.~Corollary \ref{cptembmaintheodual}.

\begin{cor}
\label{regdecompregpotdual}
Let $D(\Azs)=\H{+}{1}+\Aos\H{+}{2}$ be a regular decomposition.
Then the operator
\[
  \Pot_{\Azs} = \Qot_{\Azs} (\cAzs)^{-1}
  \colon R(\Azs) \To \H{+}{1}
\]
is well-defined and a regular potential operator associated to $ \Azs$.
\end{cor}

\begin{rem}
The decompositions \eqref{stableDecom2} and \eqref{regdecomp} translate to the identities
\begin{align*}
  \Pot_{\Ao} \Ao  + \Az \Pot_{\Az} (\id_{D(\Ao)} - \Pot_{\Ao} \Ao) & = \id_{D(\Ao)},\\
  \Qot_{\Ao} + \Az \widetilde \Qot_{\Az} & = \id_{D(\Ao)}.
\end{align*}
\end{rem}

\subsection{Sobolev Spaces}
\label{sec:Sobolev}

Let $m\in\nz$.
We have to introduce several Sobolev spaces.
Generally, spaces of scalar functions will be denoted by standard upper-case letters,
spaces of vector fields by uppercase boldface letters,
and spaces of tensor fields by calligraphic boldface letters.
For scalar spaces we have the standard notations $\L{2}{}(\om)$,
$\L{2}{\bot}(\om):=\L{2}{}(\om)\cap\reals^{\bot_{\L{2}{}(\om)}}$,
and $\H{m}{}(\om)$, as well as for vector fields
$$\vL{2}{}(\om),\quad
\vL{2}{\bot}(\om)=\vL{2}{}(\om)\cap(\reals^{3})^{\bot_{\vL{2}{}(\om)}},\quad
\vH{m}{}(\om),\quad
\vH{}{}(\rot,\om),\quad
\vH{}{}(\div,\om).$$
Corresponding Sobolev spaces of tensor fields are 
$\tL{2}{}(\om)$, $\tH{m}{}(\om)$, $\tH{}{}(\Rot,\om)$, $\tH{}{}(\Div,\om)$.
Symmetric tensor fields are introduced by
\begin{align*}
\tL{2}{\bbS}(\om)
&:=\big\{\tM\in\tL{2}{}(\om)\,:\,\tM=\tM^{\top}\big\},
&
\tH{}{\bbS}(\Rot,\om)
&:=\tH{}{}(\Rot,\om)\cap\tL{2}{\bbS}(\om),\\
\tH{m}{\bbS}(\om)
&:=\tH{m}{}(\om)\cap\tL{2}{\bbS}(\om),
&
\tH{}{\bbS}(\Div,\om)
&:=\tH{}{}(\Div,\om)\cap\tL{2}{\bbS}(\om).
\end{align*}
Furthermore, we need a few Sobolev spaces suited for the elasticity complex, namely
\begin{align*}
\vH{}{}(\na,\om)
&:=\big\{\vv\in\vL{2}{}(\om)\,:\,\na\vv\in\tL{2}{}(\om)\big\}
=\vH{1}{}(\om),\\
\vH{}{}(\symna,\om)
&:=\big\{\vv\in\vL{2}{}(\om)\,:\,\symna\vv\in\tL{2}{}(\om)\big\},\\
\tH{}{}(\RotRott,\om)
&:=\big\{\tM\in\tL{2}{}(\om)\,:\,\RotRott\tM \in\tL{2}{}(\om)\big\},\\
\tH{}{\bbS}(\RotRott,\om)
&:=\tH{}{}(\RotRott,\om)\cap\tL{2}{\bbS}(\om).
\end{align*}
Note that for $\tM\in\tH{}{\bbS}(\RotRott,\om)$ we have
$\RotRott\tM\in\tL{2}{\bbS}(\om)$, see Lemma \ref{appformulasproof} (vi').

Homogeneous boundary conditions are always introduced in the strong form
by closure of the respective test fields 
$\Cc{\infty}{}(\om)$, $\vCc{\infty}{}(\om)$, $\tCc{\infty}{}(\om)$, $\tCc{\infty}{\bbS}(\om)$ 
in the respective Sobolev norm, i.e.,
\begin{align*}
\Hc{m}{}(\om)
&:=\ol{\Cc{\infty}{}(\om)}^{\H{m}{}(\om)},
&
\vHc{m}{}(\om)
&:=\ol{\vCc{\infty}{}(\om)}^{\vH{m}{}(\om)},\\
\tHc{m}{}(\om)
&:=\ol{\tCc{\infty}{}(\om)}^{\tH{m}{}(\om)},
&
\tHc{m}{\bbS}(\om)
&:=\ol{\tCc{\infty}{\bbS}(\om)}^{\tH{m}{}(\om)},\\
\vHc{}{}(\na,\om)
&:=\ol{\vCc{\infty}{}(\om)}^{\vH{}{}(\na,\om)}
=\vHc{1}{}(\om),
&
\vHc{}{}(\symna,\om)
&:=\ol{\vCc{\infty}{}(\om)}^{\vH{}{}(\symna,\om)},\\
\vHc{}{}(\rot,\om)
&:=\ol{\vCc{\infty}{}(\om)}^{\vH{}{}(\rot,\om)},
&
\vHc{}{}(\div,\om)
&:=\ol{\vCc{\infty}{}(\om)}^{\vH{}{}(\div,\om)},\\
\tHc{}{}(\Rot,\om)
&:=\ol{\tCc{\infty}{}(\om)}^{\tH{}{}(\Rot,\om)},
&
\tHc{}{}(\Div,\om)
&:=\ol{\tCc{\infty}{}(\om)}^{\tH{}{}(\Div,\om)},\\
\tHc{}{\bbS}(\Rot,\om)
&:=\ol{\tCc{\infty}{\bbS}(\om)}^{\tH{}{}(\Rot,\om)},
&
\tHc{}{\bbS}(\Div,\om)
&:=\ol{\tCc{\infty}{\bbS}(\om)}^{\tH{}{}(\Div,\om)},\\
\tHc{}{}(\RotRott,\om)
&:=\ol{\tCc{\infty}{}(\om)}^{\tH{}{}(\RotRott,\om)},
&
\tHc{}{\bbS}(\RotRott,\om)
&:=\ol{\tCc{\infty}{\bbS}(\om)}^{\tH{}{}(\RotRott,\om)}.
\end{align*}
We use the convention $\H{}{}(\om)=\H{0}{}(\om)=\L{2}{}(\om)=\Hc{0}{}(\om)=\Hc{}{}(\om)$
and the same for the vector and tensor valued spaces.

As usual $\H{-m}{}(\om)$ denotes the dual of $\Hc{m}{}(\om)$,
and we introduce $\vH{-m}{}(\om)$, $\tH{-m}{}(\om)$, and $\tH{-m}{\bbS}(\om)$ as the dual of $\vHc{m}{}(\om)$, $\tHc{m}{}(\om)$, and $\tHc{m}{\bbS}(\om)$, respectively.

\subsection{Operators and Complexes}
\label{sec:OpSobolev}

Generally, the restriction of a differential operator $\A$ to its domain of definition of fields 
with homogeneous boundary conditions is denoted by $\mathring{\A}$, e.g., $\nac$. 
For a differential operator $\A$ acting on tensor fields, its restriction to spaces 
of symmetric tensor fields is denoted by $\A_{\bbS}$, e.g., $\RotRotst$ or $\Divcs$.

The classical primal and dual de Rham complexes (of unbounded operators) are given by
\begin{equation*} 
\begin{gathered}
\xymatrixcolsep{9ex}
\xymatrix{
\L{2}{}(\om) \ar[r]^-{\nac} &
\vL{2}{}(\om) \ar[r]^-{\rotc} &
\vL{2}{}(\om) \ar[r]^-{\divc} &
\L{2}{}(\om),
}\\
\xymatrixcolsep{9ex}
\xymatrix{
\L{2}{}(\om) & \ar[l]_-{-\div} 
\vL{2}{}(\om) & \ar[l]_{\rot} 
\vL{2}{}(\om) & \ar[l]_-{-\na} 
\L{2}{}(\om)
}
\end{gathered}
\end{equation*}
with the densely defined and closed (unbounded) linear operators 
and their adjoints 
\begin{align*}
\Az:=\nac:D(\nac)=\Hc{1}{}(\om)\subset\L{2}{}(\om)
&\to\vL{2}{}(\om);
&u&\mapsto\na u,\\
\Ao:=\rotc:D(\rotc)=\vHc{}{}(\rot,\om)\subset\vL{2}{}(\om)
&\to\vL{2}{}(\om);
&\vv&\mapsto\rot\vv,\\
\At:=\divc:D(\divc)=\vHc{}{}(\div,\om)\subset\vL{2}{}(\om)
&\to\L{2}{}(\om);
&\vv&\mapsto\div\vv,\\
\Azs=-\div:D(\div)=\vH{}{}(\div,\om)\subset\vL{2}{}(\om)
&\to\L{2}{}(\om);
&\vv&\mapsto-\div\vv,\\
\Aos=\rot:D(\rot)=\vH{}{}(\rot,\om)\subset\vL{2}{}(\om)
&\to\vL{2}{}(\om);
&\vv&\mapsto\rot\vv,\\
\Ats=-\na:D(\na)=\H{1}{}(\om)\subset\L{2}{}(\om)
&\to\vL{2}{}(\om);
&u&\mapsto-\na u.
\end{align*} 

The classical primal and dual elasticity complexes (of unbounded operators) are given by
\begin{equation*} 
\begin{gathered}
\xymatrixcolsep{9ex}
\xymatrix{
\vL{2}{}(\om) \ar[r]^-{\symnac} &
\tL{2}{\bbS}(\om) \ar[r]^-{\RotRotcst} &
\tL{2}{\bbS}(\om) \ar[r]^-{\Divcs} &
\vL{2}{}(\om),
}\\
\xymatrixcolsep{9ex}
\xymatrix{
\vL{2}{}(\om) & \ar[l]_-{-\Divs} 
\tL{2}{\bbS}(\om) & \ar[l]_{\RotRotst} 
\tL{2}{\bbS}(\om) & \ar[l]_-{-\symna} 
\vL{2}{}(\om)
}
\end{gathered}
\end{equation*}
with the densely defined and closed (unbounded) operators and there adjoints
\begin{align*}
\Az:=\symnac:D(\symnac)=\vHc{1}{}(\om)\subset\vL{2}{}(\om)
&\to\tL{2}{\bbS}(\om);
&\vv&\mapsto\symna\vv,\\
\Ao:=\RotRotcst:D(\RotRotcst)=\tHc{}{\bbS}(\RotRott,\om)\subset\tL{2}{\bbS}(\om)
&\to\tL{2}{\bbS}(\om);
&\tM&\mapsto\RotRott\tM,\\
\At:=\Divcs:D(\Divcs)=\tHc{}{\bbS}(\Div,\om)\subset\tL{2}{\bbS}(\om)
&\to\vL{2}{}(\om);
&\tN&\mapsto\Div\tN,\\
\Azs=-\Divs:D(\Divs)=\tH{}{\bbS}(\Div,\om)\subset\tL{2}{\bbS}(\om)
&\to\vL{2}{}(\om);
&\tN&\mapsto-\Div\tN,\\
\Aos=\RotRotst:D(\RotRotst)=\tH{}{\bbS}(\RotRott,\om)\subset\tL{2}{\bbS}(\om)
&\to\tL{2}{\bbS}(\om);
&\tM&\mapsto\RotRott\tM,\\
\Ats=-\symna:D(\symna)=\vH{1}{}(\om)\subset\vL{2}{}(\om)
&\to\tL{2}{\bbS}(\om);
&\vv&\mapsto-\symna\vv.
\end{align*}
Note that we have by the Korn inequalities
\begin{align*}
D(\symnac)&=\vHc{}{}(\symna,\om)=\vHc{1}{}(\om)=\vHc{}{}(\na,\om)=D(\nac),\\
D(\symna)&=\vH{}{}(\symna,\om)=\vH{1}{}(\om)=\vH{}{}(\na,\om)=D(\na).
\end{align*}

\section{The De Rham Complex}
\label{ssec:DeRham}

In this section we recall those basic properties of the de Rham complex, which are needed in the subsequent section for the analysis of the elasticity complex.

The classical primal and dual de Rham domain complexes read in full length
\begin{equation}
\label{derhamcomp2}
\begin{gathered}
\xymatrixcolsep{4ex}\xymatrixrowsep{1ex}
\xymatrix{
\{0\} \ar[r]^-{\iota_{\{0\}}} & 
\Hc{1}{}(\om) \ar[r]^-{\nac} & \vHc{}{}(\rot,\om) \ar[r]^-{\rotc} & \vHc{}{}(\div,\om) \ar[r]^-{\divc} & \L{2}{}(\om) 
\ar[r]^-{\pi_{\reals}} & \reals, 
\\
\{0\} & \ar[l]_-{\pi_{\{0\}}} 
\mathsf{L}^{2}(\om) & 
\ar[l]_-{-\div} \vH{}{}(\div,\om) & 
\ar[l]_-{\rot} \vH{}{}(\rot,\om) & 
\ar[l]_-{-\na} \H{1}{}(\om) &
\ar[l]_-{\iota_{\reals}} \reals.
}
\end{gathered}
\end{equation}
Here $\iota_{\{0\}}$, $\pi_{\{0\}}$ and $\iota_{\reals}$, $\pi_{\reals}$
denote the canonical embeddings and $\L{2}{}(\om)$-orthogonal projections 
of/on the finite-dimensional subspaces $\{0\}$ and $\reals$, respectively.
These complexes are Hilbert complexes.
The corresponding cohomology groups
$$\vHarm{}{D}(\om):=N(\rotc)\cap N(\div),\qquad
\vHarm{}{N}(\om):=N(\divc)\cap N(\rot)$$
are called Dirichlet and Neumann fields, respectively.
For a derivation of the de Rham complex by operator theoretical arguments, see \cite{paulyzulehner2020aarxiv}.

More generally, we consider the following variants of the primal and dual de Rham complexes (of unbounded operators) for $m\in\nz$ and $n\in\z$:
\begin{equation*} 
\xymatrixcolsep{7ex}\xymatrixrowsep{1ex}
\xymatrix{
\{0\} \ar[r]^-{\iota_{\{0\}}} & 
\Hc{m}{}(\om) \ar[r]^-{\nac} &
\vHc{m}{}(\om) \ar[r]^-{\rotc} &
\vHc{m}{}(\om) \ar[r]^-{\divc} &
\Hc{m}{}(\om)
\ar[r]^-{\pi_{\reals}} & \reals, 
\\
\{0\} & \ar[l]_-{\pi_{\{0\}}} 
\H{n}{}(\om) & \ar[l]_-{-\div} 
\vH{n}{}(\om) & \ar[l]_{\rot} 
\vH{n}{}(\om) & \ar[l]_-{-\na} 
\H{n}{}(\om) &
\ar[l]_-{\iota_{\reals}} \reals
}
\end{equation*}
with the associated domain complexes
\begin{equation}
\label{derhamcomplm}
\begin{gathered}
\xymatrixcolsep{6ex}\xymatrixrowsep{1ex}
\xymatrix{
\{0\} \ar[r]^-{\iota_{\{0\}}} & 
\Hc{m+1}{}(\om) \ar[r]^-{\nac} & 
\vHc{m}{}(\rot,\om) \ar[r]^-{\rotc} & 
\vHc{m}{}(\div,\om) \ar[r]^-{\divc} & 
\H{m}{}(\om) 
\ar[r]^-{\pi_{\reals}} & \reals, 
\\
\{0\} & \ar[l]_-{\pi_{\{0\}}} 
\H{n}{}(\om) & 
\ar[l]_-{-\div} \vH{n}{}(\div,\om) & 
\ar[l]_-{\rot} \vH{n}{}(\rot,\om) & 
\ar[l]_-{-\na} \H{n+1}{}(\om) &
\ar[l]_-{\iota_{\reals}} \reals
}
\end{gathered}
\end{equation}
and the Sobolev spaces
\begin{align*}
\vHc{m}{}(\rot,\om)
&:=\big\{\vv\in D(\rotc)\cap\vHc{m}{}(\om)\,:\,\rot\vv\in\vHc{m}{}(\om)\big\},\\
\vH{n}{}(\rot,\om)
&:=\big\{\vv\in \vH{n}{}(\om)\,:\,\rot\vv\in\vH{n}{}(\om)\big\},\\
\vHc{m}{}(\div,\om)
&:=\big\{\vv\in D(\divc)\cap\vHc{m}{}(\om)\,:\,\div\vv\in\Hc{m}{}(\om)\big\},\\
\vH{n}{}(\div,\om)
&:=\big\{\vv\in \vH{n}{}(\om)\,:\,\div\vv\in\H{n}{}(\om)\big\}.
\end{align*}
Note that in the definition of the spaces with boundary conditions the intersection with the respective domain of definition is superfluous for $m\ge 1$.

For the next lemma we introduce the following notation.
Similar to $\L{2}{\bot}(\om)$ we set for $m\in\nz$
$$\Hc{m}{\bot}(\om):=\Hc{m}{}(\om)\cap\reals^{\bot_{\L{2}{}(\om)}}.$$
We introduce kernel Sobolev spaces by
\begin{align*}
\vHc{m}{0}(\rot,\om)
&:=\big\{\vv\in\vHc{m}{}(\rot,\om)\,:\,\rot\vv=0\big\},
&\vH{n}{0}(\rot,\om)
&:=\big\{\vv\in\vH{n}{}(\rot,\om)\,:\,\rot\vv=0\big\},\\
\vHc{m}{0}(\div,\om)
&:=\big\{\vv\in\vHc{m}{}(\div,\om)\,:\,\div\vv=0\big\},
&
\vH{n}{0}(\div,\om)
&:=\big\{\vv\in\vH{n}{}(\div,\om)\,:\,\div\vv=0\big\}.
\end{align*}
Note that for $m=n=0$ we have
\begin{align*}
\vHc{0}{0}(\rot,\om)&=N(\rotc),
&
\vH{0}{0}(\rot,\om)&=N(\rot),\\
\vHc{0}{0}(\div,\om)&=N(\divc),
&
\vH{0}{0}(\div,\om)&=N(\div).
\end{align*}
We often skip the upper right index $0$, e.g.,
$\vHc{}{0}(\rot,\om)=\vHc{0}{0}(\rot,\om)=N(\rotc)$.

The existence of regular potential operators for the classical de Rham complex is well known.
Here we cite results from \cite{webercompmax,costabelmcintoshgenbogovskii}, 
and \cite{paulyzulehner2019a},
see also the corresponding literature in \cite{paulyzulehner2019a}.

\begin{lem}[Regular potential operators]
\label{regpotlemclasstoptriv}
Let $m\in\nz$ and $n\in\z$,
and let $\om$ be additionally topologically trivial. 
Then there exist linear and bounded operators
\begin{align*}
\Pot_{\nac}^{m}
:\vHc{m}{0}(\rot,\om)
&\To\Hc{m+1}{}(\om),
&
\na\Pot_{\nac}^{m}&=\id_{\vHc{m}{0}(\rot,\om)},\\
\Pot_{\na}^{n}
:\vH{n}{0}(\rot,\om)
&\To\H{n+1}{}(\om),
&
\na\Pot_{\na}^{n}&=\id_{\vH{n}{0}(\rot,\om)},\\
\Pot_{\rotc}^{m}
:\vHc{m}{0}(\div,\om)
&\To\vHc{m+1}{}(\om),
&
\rot\Pot_{\rotc}^{m}&=\id_{\vHc{m}{0}(\div,\om)}\\
\Pot_{\rot}^{n}
:\vH{n}{0}(\div,\om)
&\To\vH{n+1}{}(\om),
&
\rot\Pot_{\rot}^{n}&=\id_{\vH{n}{0}(\div,\om)},\\
\Pot_{\divc}^{m}
:\Hc{m}{\bot}(\om)
&\To\vHc{m+1}{}(\om),
&
\div\Pot_{\divc}^{m}&=\id_{\Hc{m}{\bot}(\om)},
\\
\Pot_{\div}^{n}
:\H{n}{}(\om)
&\To\vH{n+1}{}(\om),
&
\div\Pot_{\div}^{n}&=\id_{\H{n}{}(\om)}.
\end{align*}
These operators are regular potential operators associated to the domain complexes \eqref{derhamcomplm}.
In particular, all ranges are closed.
\end{lem}

All results of Section \ref{ssec:DeRham}
extend literally to the case of the vector de Rham complex (of unbounded operators)
\begin{equation}
\label{vectorderhamcomplm}
\begin{gathered}
\xymatrixcolsep{7ex}\xymatrixrowsep{1ex}
\xymatrix{
\{0\} \ar[r]^-{\iota_{\{\boldsymbol{0}\}}} & 
\vHc{m}{}(\om) \ar[r]^-{\nac} &
\tHc{m}{}(\om) \ar[r]^-{\Rotc} &
\tHc{m}{}(\om) \ar[r]^-{\Divc} &
\vHc{m}{}(\om)
\ar[r]^-{\pi_{\reals^3}} & \reals^3, 
\\
\{0\} & \ar[l]_-{\pi_{\{\boldsymbol{0}\}}} 
\vH{n}{}(\om) & \ar[l]_-{-\Div} 
\tH{n}{}(\om) & \ar[l]_{\Rot} 
\tH{n}{}(\om) & \ar[l]_-{-\na} 
\vH{n}{}(\om) &
\ar[l]_-{\iota_{\reals^3}} \reals^3
}
\end{gathered}
\end{equation}
with the necessary notational adaptions as formulated in Sections \ref{sec:Sobolev}, \ref{sec:OpSobolev}. Observe that here $\nac$ (resp.~$\na$) denotes the Jacobian of vector fields.

\section{The Elasticity Complex}
\label{sec:ElastCompl}

The classical primal and dual elasticity domain complexes read in full length
\begin{equation}
\label{ElastComplfull}
\begin{gathered}
\xymatrixcolsep{5ex}\xymatrixrowsep{1ex}
\xymatrix{
\{\boldsymbol{0}\} \ar[r]^-{\iota_{\{\boldsymbol{0}\}}} &
\vHc{1}{}(\om) \ar[r]^-{\symnac} &
\tHc{}{\bbS}(\RotRott,\om) \ar[r]^-{\RotRotcst} &
\tHc{}{\bbS}(\Div,\om) \ar[r]^-{\Divcs} &
\vL{2}{}(\om) \ar[r]^-{\pi_{\RM}} &
\RM,
\\
\{\boldsymbol{0}\} & \ar[l]_-{\pi_{\{\boldsymbol{0}\}}}
\vL{2}{}(\om) & \ar[l]_-{-\Divs}
\tH{}{\bbS}(\Div,\om) & \ar[l]_-{\RotRotst}
\tH{}{\bbS}(\RotRott,\om) & \ar[l]_-{-\symna}
\vH{1}{}(\om) & \ar[l]_-{\iota_{\RM}}
\RM.
}
\end{gathered}
\end{equation}
Here $\iota_{\{\boldsymbol{0}\}}$, $\pi_{\{\boldsymbol{0}\}}$ and $\iota_{\RM}$, $\pi_{\RM}$
denote the canonical embeddings and $\vL{2}{}(\om)$-orthogonal projections 
of/on the finite-dimensional subspaces $\{\boldsymbol{0}\}\subset\reals^3$ and the space of rigid motions 
$$\RM=\{\tQ\vx+\vq:\tQ\in\rttt\,\mathrm{skew},\,\vq\in\rt\},\qquad
\dim\RM=6,$$
respectively.
These complexes are Hilbert complexes.
The corresponding cohomology groups
$$\tHarm{}{\bbS,D}(\om):=N(\RotRotcst)\cap N(\Divs),\qquad
\tHarm{}{\bbS,N}(\om):=N(\Divcs)\cap N(\RotRotst)$$
will be called generalized Dirichlet and Neumann tensors
of the elasticity complex, respectively.
For a comprehensive derivation of the elasticity complex by operator theoretical arguments, see \cite{paulyzulehner2020aarxiv}.

More generally, we will discuss the following variants of the primal and dual  elasticity complexes (of unbounded operators) for $m\in\nz$ and $n\in\z$:
\begin{equation*}
\xymatrixcolsep{8ex}\xymatrixrowsep{1ex}
\xymatrix{
\{\boldsymbol{0}\} \ar[r]^-{\iota_{\{\boldsymbol{0}\}}} &
\vHc{m}{}(\om) \ar[r]^-{\symnac} &
\tHc{m}{\bbS}(\om) \ar[r]^-{\RotRotcst} &
\tHc{m}{\bbS}(\om) \ar[r]^-{\Divcs} &
\vHc{m}{}(\om) \ar[r]^-{\pi_{\RM}} &
\RM,
\\
\{\boldsymbol{0}\} & \ar[l]_-{\pi_{\{\boldsymbol{0}\}}}
\vH{n}{}(\om) & \ar[l]_-{-\Divs}
\tH{n}{\bbS}(\om) & \ar[l]_-{\RotRotst}
\tH{n}{\bbS}(\om) & \ar[l]_-{-\symna}
\vH{n}{}(\om) & \ar[l]_-{\iota_{\RM}}
\RM
}
\end{equation*}
with the associated domain complexes
\begin{equation}
\label{ElastComplfullVar1}
\begin{gathered}
\xymatrixcolsep{4ex}\xymatrixrowsep{1ex}
\xymatrix{
\{\boldsymbol{0}\} \ar[r]^-{\iota_{\{\boldsymbol{0}\}}} &
\vHc{m+1}{}(\om) \ar[r]^-{\symnac} &
\tHc{m}{\bbS}(\RotRott,\om) \ar[r]^-{\RotRotcst} &
\tHc{m}{\bbS}(\Div,\om) \ar[r]^-{\Divcs} &
\vHc{m}{}(\om) \ar[r]^-{\pi_{\RM}} &
\RM,
\\
\{\boldsymbol{0}\} & \ar[l]_-{\pi_{\{\boldsymbol{0}\}}}
\vH{n}{}(\om) & \ar[l]_-{-\Divs}
\tH{n}{\bbS}(\Div,\om) & \ar[l]_-{\RotRotst}
\tH{n}{\bbS}(\RotRott,\om) & \ar[l]_-{-\symna}
\vH{n+1}{}(\om) & \ar[l]_-{\iota_{\RM}}
\RM
}
\end{gathered}
\end{equation}
and the Sobolev spaces
\begin{align*}
\tHc{m}{\bbS}(\RotRott,\om)
&:=\big\{\tM\in D(\RotRotcst)\cap\tHc{m}{\bbS}(\om)\,:\,\RotRott\tM\in\tHc{m}{\bbS}(\om)\big\},\\
\tH{n}{\bbS}(\RotRott,\om)
&:=\big\{\tM\in \tH{n}{\bbS}(\om)\,:\,\RotRott\tM\in\tH{n}{\bbS}(\om)\big\},\\
\tHc{m}{\bbS}(\Div,\om)
&:=\big\{\tN\in D(\Divcs)\cap\tHc{m}{\bbS}(\om)\,:\,\Div\tN\in\vHc{m}{}(\om)\big\},\\
\tH{n}{\bbS}(\Div,\om)
&:=\big\{\tN\in \tH{n}{\bbS}(\om)\,:\,\Div\tN\in\vH{n}{}(\om)\big\}.
\end{align*}
Note that in the definition of the spaces $\tHc{m}{\bbS}(\RotRott,\om)$ and $\tHc{m}{\bbS}(\Div,\om)$ the intersection with the domain of definition is superfluous for $m\ge 2$ and $m\ge1$, respectively.

Additionally, we will also discuss the following variants of the elasticity complexes (of unbounded operators) for $m\in\n$ and $n\in\z$:
\begin{equation*}
\xymatrixcolsep{7ex}\xymatrixrowsep{1ex}
\xymatrix{
\{\boldsymbol{0}\} \ar[r]^-{\iota_{\{\boldsymbol{0}\}}} &
\vHc{m}{}(\om) \ar[r]^-{\symnac} &
\tHc{m}{\bbS}(\om) \ar[r]^-{\RotRotcst} &
\tHc{m-1}{\bbS}(\om) \ar[r]^-{\Divcs} &
\vHc{m-1}{}(\om) \ar[r]^-{\pi_{\RM}} &
\RM,
\\
\{\boldsymbol{0}\} & \ar[l]_-{\pi_{\{\boldsymbol{0}\}}}
\vH{n-1}{}(\om) & \ar[l]_-{-\Divs}
\tH{n-1}{\bbS}(\om) & \ar[l]_-{\RotRotst}
\tH{n}{\bbS}(\om) & \ar[l]_-{-\symna}
\vH{n}{}(\om) & \ar[l]_-{\iota_{\RM}}
\RM
}
\end{equation*}
with the associated domain complexes
\begin{equation}
\label{ElastComplfullVar2}
\begin{gathered}
\xymatrixcolsep{2.4ex}\xymatrixrowsep{1ex}
\xymatrix{
\{\boldsymbol{0}\} \ar[r]^-{\iota_{\{\boldsymbol{0}\}}} &
\vHc{m+1}{}(\om) \ar[r]^-{\symnac} &
\tHc{m,m-1}{\bbS}(\RotRott,\om) \ar[r]^-{\RotRotcst} &
\tHc{m-1}{\bbS}(\Div,\om) \ar[r]^-{\Divcs} &
\vHc{m-1}{}(\om) \ar[r]^-{\pi_{\RM}} &
\RM,
\\
\{\boldsymbol{0}\} & \ar[l]_-{\pi_{\{\boldsymbol{0}\}}}
\vH{n-1}{}(\om) & \ar[l]_-{-\Divs}
\tH{n-1}{\bbS}(\Div,\om) & \ar[l]_-{\RotRotst}
\tH{n,n-1}{\bbS}(\RotRott,\om) & \ar[l]_-{-\symna}
\vH{n+1}{}(\om) & \ar[l]_-{\iota_{\RM}}
\RM
}
\end{gathered}
\end{equation}
and the slightly less regular Sobolev spaces
\begin{align*}
\tHc{m,m-1}{\bbS}(\RotRott,\om)
&:=\big\{\tM\in D(\RotRotcst)\cap\tHc{m}{\bbS}(\om)\,:\,\RotRott\tM\in\tHc{m-1}{\bbS}(\om)\big\},\\
\tH{n,n-1}{\bbS}(\RotRott,\om)
&:=\big\{\tM\in\tH{n}{\bbS}(\om)\,:\,\RotRott\tM\in\tH{n-1}{\bbS}(\om)\big\},
\end{align*}
where again the intersection with $D(\RotRotcst)$ is not needed for $m\ge2$.

For our purposes the case $n=0$ will be most important, i.e.,
$$\tH{0,-1}{\bbS}(\RotRott,\om)
=\big\{\tM\in\tL{2}{\bbS}(\om)\,:\,\RotRott\tM\in\tH{-1}{\bbS}(\om)\big\}.$$

The first variants of the elasticity complexes are the more natural extension of the classical elasticity complex, while the second variants are better adapted to the technical difficulties due to the second order operator $\RotRott$, as we shall see.

\begin{rem}[One operator]
\label{secordop}
Note that, e.g., the second order operator $\RotRotst$ is ``one'' operator 
and not a composition of the two first order operators $\Rot$ and $\Rott$. 
Similarly, the operators $\RotRotcst$ or $\symnac$, $\symna$ 
have to be understood in the same sense as one operator.
This fact is underlined by some new results about Fredholm indices presented in \cite{PW2020a}.
The differential operators are of mixed order but cannot be seen 
as leading order type with relatively compact perturbation. 
\end{rem}

We will now proceed to show the main result of this paper, 
the compact embeddings for the elasticity complex as formulated in Theorem \ref{cptembela1}. 
The proof is based on Theorem \ref{cptembmaintheo}. 
Therefore, we have to construct appropriate regular decompositions. We do this in two main steps.
\begin{itemize}
\item
In Section \ref{ssec:primdualregpot} we start with the construction of regular potential operators for the elasticity complex in the case of topologically trivial domains (Theorem \ref{regpottheotoptriv}). This immediately implies that the elasticity complex is an exact Hilbert complex in this case (Theorem \ref{maintheo-ela1}). 
Then we will use these regular potential operators for constructing regular decompositions for the elasticity complex for topologically trivial domains (Corollary \ref{regdecocortoptriv}).
\item
In Section \ref{ssec:primdualregpotgendom} we extend the results on regular decompositions for the elasticity complex to general strong Lipschitz domains by a partition of unity argument (Theorem \ref{regdecopotgenlip}).   
Based on these results it will be shown that the elasticity complexes are compact Hilbert complexes (Theorem \ref{cptembela1}). Additionally, we construct regular potentials even in this general case (Corollary \ref{regdecopotgenlip2}) by applying Corollary \ref{regdecompregpot}.
\end{itemize}

\subsection{Topologically Trivial Domains: Regular Potentials and Decompositions}
\label{ssec:primdualregpot}

For the next theorem we introduce the following notations.
Similarly to the orthogonality to $\reals$ and $\reals^{3}$ 
indicated in the Hilbert spaces $\L{2}{\bot}(\om)=\L{2}{\bot_{\reals}}(\om)$ 
and $\vL{2}{\bot}(\om)=\vL{2}{\bot_{\reals^{3}}}(\om)$, respectively, we set
$$\vL{2}{\bot_{\RM}}(\om):=\vL{2}{}(\om)\cap\RM^{\bot_{\vL{2}{}(\om)}}.$$
Let $m\in\nz$ and $n\in\z$. We introduce kernel Sobolev spaces by
\begin{align*}
\tHc{m}{\bbS,0}(\RotRott,\om)
&:=\big\{\tM\in\tHc{m}{\bbS}(\RotRott,\om)\,:\,\RotRott\tM=0\big\},\\
\tH{n}{\bbS,0}(\RotRott,\om)
&:=\big\{\tM\in\tH{n}{\bbS}(\RotRott,\om)\,:\,\RotRott\tM=0\big\},\\
\tHc{m}{\bbS,0}(\Div,\om)
&:=\big\{\tN\in\tHc{m}{\bbS}(\Div,\om)\,:\,\Div\tN=0\big\},\\
\tH{n}{\bbS,0}(\Div,\om)
&:=\big\{\tN\in\tH{n}{\bbS}(\Div,\om)\,:\,\Div\tN=0\big\}.
\end{align*}
Note that for $m=n=0$ we have
\begin{align*}
\tHc{0}{\bbS,0}(\RotRott,\om)&=N(\RotRotcst),
&
\tH{0}{\bbS,0}(\RotRott,\om)&=N(\RotRotst),\\
\tHc{0}{\bbS,0}(\Div,\om)&=N(\Divcs),
&
\tH{0}{\bbS,0}(\Div,\om)&=N(\Divs).
\end{align*}

Using the linear and bounded regular potential operators
$\Pot_{\nac}^m$, $\Pot_{\na}^n$,
$\Pot_{\rotc}^m$, $\Pot_{\rot}^n$,
$\Pot_{\divc}^m$, $\Pot_{\div}^n$ for the de Rham complex
from Lemma \ref{regpotlemclasstoptriv}  
and their counterparts
$\Pot_{\nac}^m$, $\Pot_{\na}^n$, $\Pot_{\Rotc}^m$, 
$\Pot_{\Rot}^n$, $\Pot_{\Divc}^m$, $\Pot_{\Div}^n$ for the vector de Rham complex \eqref{vectorderhamcomplm},
we start now with our first regular potentials for the elasticity complex.

\begin{theo}[Regular potential operators]
\label{regpottheotoptriv}
Let $\om$ be additionally topologically trivial
and let $m\in\nz$ and $n\in\z$. 
Then there exist linear and bounded operators
\begin{align*}
\Pot_{\symnac}^{m}
:\tHc{m}{\bbS,0}(\RotRott,\om)
&\To\vHc{m+1}{}(\om),
&
\symna\Pot_{\symnac}^{m}
&=\id_{\tHc{m}{\bbS,0}(\RotRott,\om)},\\
\Pot_{\symna}^{n}
:\tH{n}{\bbS,0}(\RotRott,\om)
&\To\vH{n+1}{}(\om),
&
\symna\Pot_{\symna}^{n}
&=\id_{\tH{n}{\bbS,0}(\RotRott,\om)},\\
\Pot_{\RotRotcst}^{m}
:\tHc{m}{\bbS,0}(\Div,\om)
&\To\tHc{m+2}{\bbS}(\om),
&
\RotRott\Pot_{\RotRotcst}^{m}
&=\id_{\tHc{m}{\bbS,0}(\Div,\om)},\\
\Pot_{\RotRotst}^{n}
:\tH{n}{\bbS,0}(\Div,\om)
&\To\tH{n+2}{\bbS}(\om),
&
\RotRott\Pot_{\RotRotst}^{n}
&=\id_{\tH{n}{\bbS,0}(\Div,\om)},\\
\Pot_{\Divcs}^{m}
:\vHc{m}{\bot_{\RM}}(\om)
&\To\tHc{m+1}{\bbS}(\om),
&
\Div\Pot_{\Divcs}^{m}
&=\id_{\vHc{m}{\bot_{\RM}}(\om)},\\
\Pot_{\Divs}^{n}
:\vH{n}{}(\om)
&\To\tH{n+1}{\bbS}(\om),
&
\Div\Pot_{\Divs}^{n}
&=\id_{\vH{n}{}(\om)}
\end{align*}
with
\begin{align*}
\Pot_{\symnac}^{m} \tM
&=\Pot_{\nac}^{m}\big(\tM+\spn\Pot_{\nac}^{m-1}\Rott \tM \big) \qquad \text{for} \ m\geq2,\\
\Pot_{\symna}^{n} \tM
&=\Pot_{\na}^{n}\big(\tM+\spn\Pot_{\na}^{n-1}\Rott \tM \big),\\
\Pot_{\RotRotcst}^{m} \tN
&=\sym\Pot_{\Rotc}^{m+1}\big( (\Pot_{\Rotc}^{m}\tN)^{\top}-(\tr \Pot_{\Rotc}^{m}\tN) \tI \big),\\
\Pot_{\RotRotst}^{n} \tN
&=\sym\Pot_{\Rot}^{n+1}\big( (\Pot_{\Rot}^{n}\tN)^{\top}- (\tr \Pot_{\Rot}^{n}\tN) \tI \big),\\
\Pot_{\Divcs}^{m} \vv
&=\sym\big( \Pot_{\Divc}^{m} \vv-2\Rott ( \Pot_{\Divc}^{m+1} \spn^{-1}\skw \Pot_{\Divc}^{m} \vv ) \big),\\
\Pot_{\Divs}^{n} \vv
&=\sym\big( \Pot_{\Div}^{n} \vv-2\Rott ( \Pot_{\Div}^{n+1} \spn^{-1}\skw \Pot_{\Div}^{n} \vv ) \big).
\end{align*}
\end{theo}

\begin{proof}
First, let us consider only $m=n\in\nz$.

$\bullet$ 
Let $\tN\in\tHc{m}{\bbS,0}(\Div,\om)$ resp. $\tN\in\tH{m}{\bbS,0}(\Div,\om)$.
Then, with row-wise applications of Lemma \ref{regpotlemclasstoptriv},
there exists $\tE = \Pot_{\Rotc}^{m} \tN \in\tHc{m+1}{}(\om)$ resp. $\tE=\Pot_{\Rot}^{m}\tN\in\tH{m+1}{}(\om)$ such that
$\Rot\tE=\tN$. As $\tN$ is symmetric we get with Lemma \ref{appformulasproof} (v) 
for $\widetilde\tE:=\tE^{\top}-(\tr\tE)\tI\in\tHc{m+1}{}(\om)$ resp. $\tH{m+1}{}(\om)$
$$\Div\widetilde\tE
=\Div\tE^{\top}-\na\tr\tE
=2\spn^{-1}\skw\Rot\tE
=2\spn^{-1}\skw\tN=0.$$
Again row-wise applications of Lemma \ref{regpotlemclasstoptriv}
show the existence of $\widetilde\tM=\Pot_{\Rotc}^{m+1}\widetilde\tE\in\tHc{m+2}{}(\om)$ 
resp. $\widetilde\tM=\Pot_{\Rot}^{m+1}\widetilde\tE\in\tH{m+2}{}(\om)$ such that
$\Rot\widetilde\tM=\widetilde\tE$. Hence
$$\tN
=\Rot\tE
=\Rot\big(\widetilde\tE^{\top}+(\tr\tE)\tI\big)
=\Rot\Rott\widetilde\tM+\Rot\big((\tr\tE)\tI\big).$$
Lemma \ref{appformulasproof} (i) yields 
$\Rot\big((\tr\tE)\tI\big)
=-\spn\na\tr\tE$, 
which is skew-symmetric.
Therefore, 
$$\tN
=\sym\tN
=\sym\RotRott\widetilde\tM
=\RotRott\sym\widetilde\tM$$
see Lemma \ref{appformulasproof} (vi).
The tensor field $\tM:=\sym\widetilde\tM\in\tHc{m+2}{\bbS}(\om)$ resp. $\tM:=\sym\widetilde\tM\in\tH{m+2}{\bbS}(\om)$ 
solves $\RotRott\tM=\tN$ and is given by
$$\tM
=\sym\Pot^{m+1}\widetilde\tE
=\sym\Pot^{m+1}\big((\Pot^{m}\tN)^{\top}-(\tr\Pot^{m}\tN)\tI\big)
=:\widehat\Pot^{m}\tN,$$ 
where $\Pot^{k}\in\{\Pot_{\Rotc}^{k},\Pot_{\Rot}^{k}\}$
and $\widehat\Pot^{m}\in\{\Pot_{\RotRotcst}^{m},\Pot_{\RotRotst}^{m}\}$.

$\bullet$ 
Let $\vv\in\vHc{m}{\bot_{\RM}}(\om) (\subset \vL{2}{\bot_{\RM}}(\om))$ resp. $\vv\in\vH{m}{}(\om)$.
Row-wise applications of Lemma \ref{regpotlemclasstoptriv}
show the existence of $\tE=\Pot_{\Divcs}^{m} \vv\in\tHc{m+1}{}(\om)$ resp. $\tE=\Pot_{\Divc}^{m} \vv \in \tH{m+1}{}(\om)$ such that
$\Div\tE=\vv$.
Moreover, for $\vv\in\vL{2}{\bot_{\RM}}(\om)$, $\tE\in\tHc{m+1}{}(\om)$, 
and $\vb\in\rt$ we have that $\tA:=\spn\vb$ is skew-symmetric
and hence 
\begin{align*}
0&=\scp{\vv}{\tA\vx}_{\vL{2}{}(\om)}
=\scp{\Div\tE}{\tA\vx}_{\vL{2}{}(\om)}
=-\scp{\tE}{\na(\tA\vx)}_{\tL{2}{}(\om)}\\
&=-\scp{\tE}{\tA}_{\tL{2}{}(\om)}
=-\scp{\skw\tE}{\spn\vb}_{\tL{2}{}(\om)}
=-2\scp{\spn^{-1}\skw\tE}{\vb}_{\vL{2}{}(\om)},
\end{align*}
using the relation $(\spn\vd)\!:\!(\spn\vb)=2\,\vd\cdot\vb$ for $\vd = \spn^{-1} \skw \tE$, i.e., 
$\spn^{-1}\skw\tE\in\vHc{m+1}{\bot}(\om)$
resp. $\spn^{-1}\skw\tE\in\vH{m+1}{}(\om)$.
By Lemma \ref{regpotlemclasstoptriv} there is some
$\widetilde\tE=\Pot_{\Divc}^{m+1}\spn^{-1}\skw\tE\in\tHc{m+2}{}(\om)$ 
resp. $\widetilde\tE=\Pot_{\Div}^{m+1}\spn^{-1}\skw\tE\in\tH{m+2}{}(\om)$ 
such that $\Div\widetilde\tE=\spn^{-1}\skw\tE$. 
By Lemma \ref{appformulasproof} (iii), (v) we see
$$\Div\skw\tE
=-\rot\spn^{-1}\skw\tE
=-\rot\Div\widetilde\tE
=-2\Div\sym\Rot\widetilde\tE^{\top}.$$
Hence
\begin{align*}
\vv
&=\Div\tE
=\Div\sym\tE+\Div\skw\tE
=\Div\tN
\end{align*}
with $\tN:=\sym\big(\tE-2\Rot\widetilde\tE^{\top}\big)\in\tHc{m+1}{\bbS}(\om)$
resp. $\tH{m+1}{\bbS}(\om)$ given by
\begin{align*}
\tN
&=\sym\big(\tE-2\Rot(\Pot^{m+1}\spn^{-1}\skw\tE)^{\top}\big)\\
&=\sym\big(\Pot^{m}\vv-2\Rott(\Pot^{m+1}\spn^{-1}\skw\Pot^{m}\vv)\big)
=:\widehat\Pot^{m}\vv,
\end{align*}
where $\Pot^{k}\in\{\Pot_{\Divc}^{k},\Pot_{\Div}^{k}\}$
and $\widehat\Pot\in\{\Pot_{\Divcs}^{m},\Pot_{\Divs}^{m}\}$.

$\bullet$ 
Let $\tM\in\tH{m}{\bbS,0}(\RotRott,\om)$.
Then $\Rott\tM\in\tH{m-1}{}(\om)$ with $\Rot(\Rott\tM)=0$.
Row-wise applications of Lemma \ref{regpotlemclasstoptriv}
yield $\vu=\Pot_{\na}^{m-1}\Rott\tM\in\vH{m}{}(\om)$ such that $\na\vu=\Rott\tM$.
By Lemma \ref{appformulasproof} (ii) we see $\div\vu=\tr\na\vu=\tr\Rott\tM=0$
as $\tM$ is symmetric and thus by using Lemma \ref{appformulasproof} (iv) we get
$\na\vu=-\Rott\spn\vu$. Therefore, $\Rott(\tM+\spn\vu)=0$ and again
Lemma \ref{regpotlemclasstoptriv} gives some $\vv=\Pot_{\na}^{m} (\tM+\spn\vu)\in\vH{m+1}{}(\om)$
with $\na\vv= \tM + \spn \vu$. As $\tM$ is symmetric we obtain $\tM=\sym\na\vv$ and
$$\vv
=\Pot_{\na}^{m}(\tM+\spn\vu)
=\Pot_{\na}^{m}(\tM+\spn\Pot_{\na}^{m-1}\Rott\tM)
=:\Pot_{\symna}^{m}\tM.$$ 
Now, let $\tM\in\tHc{m}{\bbS,0}(\RotRott,\om)$.
We can extend $\tM$ by zero to 
$\widetilde\tM\in\tHc{m}{\bbS,0}(\RotRott,B)$ (by definition),
where $B$ denotes an open ball containing $\ol{\om}$.
This can be seen 
by approximating $\tM$ by a sequence $(\tM_{n})\subset\tCc{\infty}{\bbS}(\om)$
of test tensors with $\tM_{n}\to\tM$ in $\tH{}{\bbS}(\RotRott,\om)$ 
resp. $\tH{m}{}(\om)$ and 
observing $\widetilde\tM_{n}\to\widetilde\tM$ in $\tH{}{\bbS}(\RotRott,B)$
resp. $\tH{m}{}(B)$.
In particular, we have $\widetilde\tM\in\tH{m}{\bbS,0}(\RotRott,B)$
and from the latter arguments 
(of the proof without boundary conditions) we obtain 
$\widetilde\vv:=\Pot_{\symna,B}^{m}\widetilde\tM\in\vH{m+1}{}(B)$
with $\sym\na\widetilde\vv=\widetilde\tM$ in $B$. 
As $\widetilde\tM$ vanishes in the complement of $\om$,
$\widetilde\vv$ is a rigid motion $\vr_{\widetilde\vv}\in\RM$ in $B\setminus\ol{\om}$.
Thus we have $\vv:=\widetilde\vv-\vr_{\widetilde\vv}\in\vH{m+1}{}(B)$ and 
$\vv|_{B\setminus\ol{\om}}=0$, which implies 
$\vv\in\tHc{m+1}{}(\om)$ and $\sym\na\vv=\sym\na\widetilde\vv=\tM$ in $\om$.
Moreover,
$$\vv
=(\Pot_{\symna,B}^{m}\widetilde\tM)|_{\om}-\vr_{\widetilde\vv}
=:\Pot_{\symnac}^{m}\tM.$$ 
For $m\geq2$ we can
literally following the corresponding proof without boundary conditions as well.
In particular, we get $\Rott\tM\in\tHc{m-1}{0}(\Rot,\om)$,
$\vu\in\vHc{m}{}(\om)$, and $\tM-\spn\vu\in\tHc{m}{}(\om)$. Thus
$$\vHc{m+1}{}(\om)
\ni\vv
=\Pot_{\nac}^{m}(\tM+\spn\vu)
=\Pot_{\nac}^{m}(\tM+\spn\Pot_{\nac}^{m-1}\Rott\tM)
=:\Pot_{\symnac}^{m}\tM.$$ 

$\bullet$ 
Finally, let $-n\in\n$.
Then for $\tM\in\tH{n}{\bbS,0}(\RotRott,\om)$, $\tN\in\tH{n}{\bbS,0}(\Div,\om)$, 
and $\vv\in\vH{n}{}(\om)$ we literally follow the latter proofs (without boundary condition).
This concludes the proof.
\end{proof}

As a simple consequence of Theorem \ref{regpottheotoptriv} we obtain:

\begin{cor}[Regular potentials]
\label{cor1}
Let $\om$ be additionally topologically trivial
and let $m\in\nz$ and $n\in\z$. 
Then 
\begin{align*}
  \tHc{m}{\bbS,0}(\RotRott,\om)
    & = \symna \vHc{m+1}{}(\om),\\
  \tH{n}{\bbS,0}(\RotRott,\om)
    & = \symna \vH{n+1}{}(\om),\\
  \tHc{m}{\bbS,0}(\Div,\om)
    & = \RotRott \tHc{m}{\bbS}(\RotRott,\om)\\
    & = \RotRott \tHc{m+1,m}{\bbS}(\RotRott,\om)
      = \RotRott \tHc{m+2}{\bbS}(\om),\\
  \tH{n}{\bbS,0}(\Div,\om)
    & = \RotRott \tH{n}{\bbS}(\RotRott,\om)\\
    & = \RotRott \tH{n+1,n}{\bbS}(\RotRott,\om)
      = \RotRott \tH{n+2}{\bbS}(\om),\\
  \vHc{m}{\bot_{\RM}}(\om)
    & = \Div \tHc{m}{\bbS}(\Div,\om)
      = \Div \tHc{m+1}{\bbS}(\om),\\
 \vH{n}{}(\om)
   & = \Div \tH{n}{\bbS}(\Div,\om)
     = \Div \tH{n+1}{\bbS}(\om)
\end{align*}
hold. Moreover, the operators 
$\Pot_{\symnac}^{m}$, $\Pot_{\symna}^n$, $\Pot_{\RotRotcst}^m$, 
$\Pot_{\RotRotst}^n$, $\Pot_{\Divcs}^m$, $\Pot_{\Divs}^n$ 
are regular potential operators associated to the domain complexes \eqref{ElastComplfullVar1} 
and, for $m\ge 1$, the operators $\Pot_{\symnac}^{m}$, $\Pot_{\symna}^n$, 
$\Pot_{\RotRotcst}^{m-1}$, $\Pot_{\RotRotst}^{n-1}$, $\Pot_{\Divcs}^{m-1}$, $\Pot_{\Divs}^{n-1}$ 
are regular potential operators associated to the domain complexes \eqref{ElastComplfullVar2}.
In particular, all ranges are closed.
\end{cor}

\begin{proof}
We have by Theorem \ref{regpottheotoptriv}, e.g.,
\begin{align*}
\tHc{m}{\bbS,0}(\Div,\om)
&=\RotRott\Pot_{\RotRotcst}^{m}\tHc{m}{\bbS,0}(\Div,\om)
\subset\RotRott\tHc{m+2}{\bbS}(\om)\\
&\subset\RotRott\tHc{m+1,m}{\bbS}(\om)
\subset\RotRott\tHc{m}{\bbS}(\RotRott,\om) 
\subset\tHc{m}{\bbS,0}(\Div,\om),
\end{align*}
which implies the third identity. The other identities follow completely analogously.
The regularity follows from Rellich's selection theorem.
\end{proof}

Corollary \ref{cor1} immediately leads to the following result.

\begin{theo}[Closed Hilbert complexes]
\label{maintheo-ela1}
Let $m\in\nz$ and $n\in\z$ and let 
$\om$ be additionally topologically trivial.
The domain complexes of elasticity, cf.~\eqref{ElastComplfullVar1}, \eqref{ElastComplfullVar2},
\begin{equation*}
\xymatrixcolsep{5ex}\xymatrixrowsep{1ex}
\xymatrix{
\{\boldsymbol{0}\} \ar[r]^-{\iota_{\{\boldsymbol{0}\}}} &
\vHc{m+1}{}(\om) \ar[r]^-{\symnac} &
\tHc{m}{\bbS}(\RotRott,\om) \ar[r]^-{\RotRotcst} &
\tHc{m}{\bbS}(\Div,\om) \ar[r]^-{\Divcs} &
\vHc{m}{}(\om) \ar[r]^-{\pi_{\RM}} &
\RM,
\\
\{\boldsymbol{0}\} & \ar[l]_-{\pi_{\{\boldsymbol{0}\}}}
\vH{n}{}(\om) & \ar[l]_-{-\Divs}
\tH{n}{\bbS}(\Div,\om) & \ar[l]_-{\RotRotst}
\tH{n}{\bbS}(\RotRott,\om) & \ar[l]_-{-\symna}
\vH{n+1}{}(\om) & \ar[l]_-{\iota_{\RM}}
\RM,
}
\end{equation*}
and, for $m \ge 1$,
\begin{equation*}
\xymatrixcolsep{3ex}\xymatrixrowsep{1ex}
\xymatrix{
\{\boldsymbol{0}\} \ar[r]^-{\iota_{\{\boldsymbol{0}\}}} &
\vHc{m+1}{}(\om) \ar[r]^-{\symnac} &
\tHc{m,m-1}{\bbS}(\RotRott,\om) \ar[r]^-{\RotRotcst} &
\tHc{m-1}{\bbS}(\Div,\om) \ar[r]^-{\Divcs} &
\vHc{m-1}{}(\om) \ar[r]^-{\pi_{\RM}} &
\RM,
\\
\{\boldsymbol{0}\} & \ar[l]_-{\pi_{\{\boldsymbol{0}\}}}
\vH{n-1}{}(\om) & \ar[l]_-{-\Divs}
\tH{n-1}{\bbS}(\Div,\om) & \ar[l]_-{\RotRotst}
\tH{n,n-1}{\bbS}(\RotRott,\om) & \ar[l]_-{-\symna}
\vH{n+1}{}(\om) & \ar[l]_-{\iota_{\RM}}
\RM
}
\end{equation*}
are exact Hilbert complexes, all ranges are \emph{closed}, and the operators from Theorem \ref{regpottheotoptriv} 
are associated regular potential operators.
\end{theo}

\begin{rem}[Topology]
\label{regpotremtoptriv}
A closer inspection of Theorem \ref{regpottheotoptriv} 
and Corollary \ref{cor1}
together with their proofs shows that 
the assumptions on the topologies can be weakened.
In particular, in Theorem \ref{regpottheotoptriv} 
the following holds true for the potential operators: 
\begin{itemize}
\item[\bf(i)]
$\Pot_{\divc}^{m}$, $\Pot_{\div}^{n}$, $\Pot_{\Divc}^{m}$, $\Pot_{\Div}^{n}$
and hence $\Pot_{\Divcs}^{m}$, $\Pot_{\Divs}^{n}$
do not depend on the topology 
and all results involving these divergence operators
hold for general bounded strong Lipschitz domains.
\item[\bf(ii)]
$\Pot_{\nac}^{m}$, $\Pot_{\rot}^{n}$, $\Pot_{\Rot}^{n}$
and hence $\Pot_{\symnac}^{m}$, $\Pot_{\RotRotst}^{n}$
exist on the respective kernels if and only if
$$\vHarm{}{D}(\om)
=\{\boldsymbol{0}\}
\qequi
\ga\text{ is connected}.$$
\item[\bf(iii)]
$\Pot_{\na}^{n}$, $\Pot_{\rotc}^{m}$, $\Pot_{\Rotc}^{m}$
and hence $\Pot_{\symna}^{n}$, $\Pot_{\RotRotcst}^{m}$
exist on the respective kernels if and only if
$$\vHarm{}{N}(\om)
=\{\boldsymbol{0}\}
\qequi
\om\text{ has no handles}.$$
\end{itemize}
As the assertions for the potentials imply
the triviality of the cohomology groups we have:
\begin{itemize}
\item[\bf(iv)]
$\tHarm{}{\bbS,D}(\om)
=\{\boldsymbol{0}\}
\qequi
\vHarm{}{D}(\om)
=\{\boldsymbol{0}\}
\qequi
\ga\text{ is connected}$
\item[\bf(v)]
$\tHarm{}{\bbS,N}(\om)
=\{\boldsymbol{0}\}
\qequi
\vHarm{}{N}(\om)
=\{\boldsymbol{0}\}
\qequi
\om\text{ has no handles}$
\end{itemize}
\end{rem}

From Lemma \ref{lem:SpaceRegDeco}, Lemma \ref{lem:SpaceRegDecorefined}, 
Theorem \ref{regpottheotoptriv}, and Corollary \ref{cor1} 
we immediately obtain the following regular decompositions.

\begin{cor}[Regular decompositions]
\label{regdecocortoptriv}
Let $\om$ be additionally topologically trivial and let $m\in\nz$ and $n\in\z$. 
Then the regular decompositions 
\begin{align*}
\tHc{m}{\bbS}(\RotRott,\om)
&=R(\Pot_{\RotRotcst}^{m})
\dotplus\symna R(\Pot_{\symnac}^{m})\\
&=R(\Pot_{\RotRotcst}^{m})
\dotplus\symna\vHc{m+1}{}(\om)
=\tHc{m+2}{\bbS}(\om)+\symna\vHc{m+1}{}(\om),\\
\tH{n}{\bbS}(\RotRott,\om)
&=R(\Pot_{\RotRotst}^{n})
\dotplus\symna R(\Pot_{\symna}^{n})\\
&=R(\Pot_{\RotRotst}^{n})
\dotplus\symna\vH{n+1}{}(\om)
=\tH{n+2}{\bbS}(\om)+\symna\vH{n+1}{}(\om),\\
\tHc{m}{\bbS}(\Div,\om)
&=R(\Pot_{\Divcs}^{m})
\dotplus\RotRott R(\Pot_{\RotRotcst}^{m})\\
&=R(\Pot_{\Divcs}^{m})
\dotplus\RotRott\tHc{m+2}{\bbS}(\om)
=\tHc{m+1}{\bbS}(\om)+\RotRott\tHc{m+2}{\bbS}(\om),\\
\tH{n}{\bbS}(\Div,\om)
&=R(\Pot_{\Divs}^{n})
\dotplus\RotRott R(\Pot_{\RotRotst}^{n})\\
&=R(\Pot_{\Divs}^{n})
\dotplus\RotRott\tH{n+2}{\bbS}(\om)
=\tH{n+1}{\bbS}(\om)+\RotRott\tH{n+2}{\bbS}(\om)
\end{align*}
hold. 
\end{cor}

By similar arguments we also obtain the following non-standard regular decompositions.

\begin{cor}[Regular decompositions]
\label{regdecononstandcor}
Let $\om$ be additionally topologically trivial and 
let $m\in\n$ and $n\in\z$. Then the regular decompositions
\begin{align*}
\tHc{m,m-1}{\bbS}(\RotRott,\om)
&=R(\Pot_{\RotRotcst}^{m-1})
\dotplus\symna R(\Pot_{\symnac}^{m})\\
&=R(\Pot_{\RotRotcst}^{m-1})
\dotplus\symna\vHc{m+1}{}(\om)
=\tHc{m+1}{\bbS}(\om)+\symna\vHc{m+1}{}(\om),\\
\tH{n,n-1}{\bbS}(\RotRott,\om)
&=R(\Pot_{\RotRotst}^{n-1})
\dotplus\symna R(\Pot_{\symna}^{n})\\
&=R(\Pot_{\RotRotst}^{n-1})
\dotplus\symna\vH{n+1}{}(\om)
=\tH{n+1}{\bbS}(\om)+\symna\vH{n+1}{}(\om)
\end{align*}
hold. 
\end{cor}

\subsection{General Strong Lipschitz Domains: Regular Decompositions and  Potentials}
\label{ssec:primdualregpotgendom}

To prove regular potentials and regular decompositions 
as well as more sophisticated compact embeddings 
for general bounded strong Lipschitz domains of arbitrary topology
we start with a cutting lemma, showing a technical difficulty 
due to the ``second order'' nature of the operator $\RotRott$.

\begin{lem}[Cutting lemma]
\label{cutlem}
Let $\varphi\in\Cc{\infty}{}(\rt)$ and let $m\in\nz$ and $n\in\z$.
\begin{itemize}
\item[\bf(i)]
If $\tN\in\tH{n}{\bbS}(\Div,\om)$, then
$\varphi\tN\in\tH{n}{\bbS}(\Div,\om)$ and 
\begin{align}
\label{Divformula}
\Div(\varphi\tN)
=\varphi\Div\tN+\tN\na\varphi.
\end{align}
\item[\bf(i')]
If $\tN\in\tHc{m}{\bbS}(\Div,\om)$, then
$\varphi\tN\in\tHc{m}{\bbS}(\Div,\om)$ and \eqref{Divformula} holds.
\item[\bf(ii)]
If $\tM\in\tH{n,n-1}{\bbS}(\RotRott,\om)$, then 
$\varphi\tM\in\tH{n,n-1}{\bbS}(\RotRott,\om)$ and 
\begin{align}
\label{RotRotTformula}
\RotRott(\varphi\tM)
=\varphi\RotRott\tM
+2\sym\big((\spn\na\varphi)\Rot\tM\big)
+\Psi(\na\na \varphi,\tM)
\end{align}
holds with an algebraic operator $\Psi$.
In particular, this holds for $\tM\in\tH{n}{\bbS}(\RotRott,\om)$.
\item[\bf(ii')]
Let $m\geq1$ and let $\tM\in\tHc{m,m-1}{\bbS}(\RotRott,\om)$. Then
$\varphi\tM\in\tHc{m,m-1}{\bbS}(\RotRott,\om)$ together with 
\eqref{RotRotTformula}.
In particular, this holds for $\tM\in\tHc{m}{\bbS}(\RotRott,\om)$.
\end{itemize}
\end{lem}

\begin{proof}
(i) and (i') follow immediately from row-wise applications
of the corresponding results for vector fields.
Let us consider (ii). 
For vectors $\vp,\vq\in\rt$ and a matrix $\tQ\in\rttt$ 
it holds $\vp\times\vq=(\spn\vp)\vq$ and 
$\vp\times\tQ=-\tQ\spn\vp$ with row-wise operation of the exterior product.
Moreover, for $\tP,\tQ=\big( \vq_1,\vq_2,\vq_3\big)\in\tCc{\infty}{}(\rt)$ we have
$$\Rot(\tP\tQ)=\tP\Rot\tQ+(\na\tP) \times \tQ$$
with
$(\na\tP) \times \tQ
   := \big( 
       (\partial_2 \tP) \vq_3 - (\partial_3 \tP) \vq_2, 
       (\partial_3 \tP) \vq_1 - (\partial_1 \tP) \vq_3, 
       (\partial_1 \tP) \vq_2 - (\partial_2 \tP) \vq_1 \big)$.
Now, let $\tM\in\tCc{\infty}{}(\rt)$. 
By the standard row-wise formula for vector fields we compute
\begin{align*}
\Rot(\varphi\tM)
&=\varphi\Rot\tM
-\tM\spn\na\varphi,\qquad
\Rott(\varphi\tM)
=\varphi\Rott\tM
+(\spn\na\varphi)\tM^{\top}
\end{align*}
and
\begin{align*}
&\qquad\RotRott(\varphi\tM)\\
&=\varphi\RotRott\tM
-\Rott\tM\spn\na\varphi
+(\spn\na\varphi)\Rot\tM^{\top}
+\big(\na(\spn\na\varphi)\big) \times \tM^{\top}\\
&=\varphi\RotRott\tM
+\big((\spn\na\varphi)\Rot\tM\big)^{\top}
+(\spn\na\varphi)\Rot\tM^{\top}
+\Psi(\na\na\varphi,\tM).
\end{align*}
If $\tM$ is symmetric, we observe
$$\RotRott(\varphi\tM)
=\varphi\RotRott\tM
+2\sym\big((\spn\na\varphi)\Rot\tM\big)
+\Psi(\na\na\varphi,\tM).$$
Finally, the latter formulas extend to distributions as well.
\end{proof}

Now we proceed by showing regular decompositions for the elasticity complexes.

\begin{theo}[Regular decompositions]
\label{regdecopotgenlip}
Let $m,n\in\nz$. 
The regular decompositions in Corollary \ref{regdecocortoptriv} 
and Corollary \ref{regdecononstandcor} hold in the sense that 
there exist bounded linear operators
\begin{align*}
\Qot_{\RotRotcst}^{m,m-1}
:\tHc{m,m-1}{\bbS}(\RotRott,\om)
&\To\tHc{m+1}{\bbS}(\om),
&
\widetilde{\Qot}_{\symnac}^{m,m-1}
:\tHc{m,m-1}{\bbS}(\RotRott,\om)
&\To\vHc{m+1}{}(\om)
\intertext{for $m\ge 1$, and}
\Qot_{\RotRotst}^{n,n-1}
:\tH{n,n-1}{\bbS}(\RotRott,\om)
&\To\tH{n+1}{\bbS}(\om),
&
\widetilde{\Qot}_{\symna}^{n,n-1}
:\tH{n,n-1}{\bbS}(\RotRott,\om)
&\To\vH{n+1}{}(\om),\\
\Qot_{\Divcs}^{m}
:\tHc{m}{\bbS}(\Div,\om)
&\To\tHc{m+1}{\bbS}(\om),
&
\widetilde{\Qot}_{\RotRotcst}^{m}
:\tHc{m}{\bbS}(\Div,\om)
&\To\tHc{m+2}{\bbS}(\om),\\
\Qot_{\Divs}^{n}
:\tH{n}{\bbS}(\Div,\om)
&\To\tH{n+1}{\bbS}(\om),
&
\widetilde{\Qot}_{\RotRotst}^{n}
:\tH{n}{\bbS}(\Div,\om)
&\To\tH{n+2}{\bbS}(\om),
\end{align*}
such that
\begin{align*}
\tHc{m,m-1}{\bbS}(\RotRott,\om)
&=\tHc{m+1}{\bbS}(\om)+\symna\vHc{m+1}{}(\om)
=R(\Qot_{\RotRotcst}^{m,m-1})
+\symna R(\widetilde{\Qot}_{\symnac}^{m,m-1})\\
&=R(\Qot_{\RotRotcst}^{m,m-1})
+\tHc{m}{\bbS,0}(\RotRott,\om)
\intertext{for $m\ge 1$, and}
\tH{n,n-1}{\bbS}(\RotRott,\om)
&=\tH{n+1}{\bbS}(\om)+\symna\vH{n+1}{}(\om)
=R(\Qot_{\RotRotst}^{n,n-1})
+\symna R(\widetilde{\Qot}_{\symna}^{n,n-1})\\
&=R(\Qot_{\RotRotst}^{n,n-1})
+\tH{n}{\bbS,0}(\RotRott,\om),\\
\tHc{m}{\bbS}(\Div,\om)
&=\tHc{m+1}{\bbS}(\om)+\RotRott\tHc{m+2}{\bbS}(\om)
=R(\Qot_{\Divcs}^{m})
+\RotRott R(\widetilde{\Qot}_{\RotRotcst}^{m})\\
&=R(\Qot_{\Divcs}^{m})
+\tHc{m}{\bbS,0}(\Div,\om),\\
\tH{n}{\bbS}(\Div,\om)
&=\tH{n+1}{\bbS}(\om)+\RotRott\tH{n+2}{\bbS}(\om)
=R(\Qot_{\Divs}^{n})
+\RotRott R(\widetilde{\Qot}_{\RotRotst}^{n})\\
&=R(\Qot_{\Divs}^{n})
+\tH{n}{\bbS,0}(\Div,\om).
\end{align*}
\end{theo}

\begin{proof}
Let $(\varphi_{\ell},U_{\ell})$ be the partition of unity 
together with $\phi_{\ell}\in\Cc{\infty}{}(U_{\ell})$ from Section \ref{sec:domains}.

$\bullet$ Let $\tN\in\tHc{m}{\bbS}(\Div,\om)$. 
By Lemma \ref{cutlem}, approximation,
and Corollary \ref{regdecocortoptriv} we have
$$\varphi_{\ell}\tN
\in\tHc{m}{\bbS}(\Div,\om_{\ell})
=\tHc{m+1}{\bbS}(\om_{\ell})+\RotRott\tHc{m+2}{\bbS}(\om_{\ell}).$$
Thus $\varphi_{\ell}\tN=\tN_{\ell}+\RotRott\tM_{\ell}$
with $\tN_{\ell}\in\tHc{m+1}{\bbS}(\om_{\ell})$ and $\tM_{\ell}\in\tHc{m+2}{\bbS}(\om_{\ell})$.
Extending $\tN_{\ell}$ and $\tM_{\ell}$ by zero gives
$\widetilde\tN_{\ell}\in\tHc{m+1}{\bbS}(\om)$ and $\widetilde\tM_{\ell}\in\tHc{m+2}{\bbS}(\om)$ with 
$$\tN
=\sum_{\ell=1}^{L}\varphi_{\ell}\tN
=\sum_{\ell=1}^{L}\widetilde\tN_{\ell}+\RotRott\sum_{\ell=1}^{L}\widetilde\tM_{\ell}
\in\tHc{m+1}{\bbS}(\om)+\RotRott\tHc{m+2}{\bbS}(\om),$$
and all applied operations are continuous, showing the existence 
of $\Qot_{\Divcs}^{m}$ and $\widetilde{\Qot}_{\RotRotcst}^{m}$.

$\bullet$ Let $\tN\in\tH{n}{\bbS}(\Div,\om)$. 
As before we get
$$\varphi_{\ell}\tN
\in\tH{n}{\bbS}(\Div,\om_{\ell})
=\tH{n+1}{\bbS}(\om_{\ell})+\RotRott\tH{n+2}{\bbS}(\om_{\ell})$$
and $\varphi_{\ell}\tN=\tN_{\ell}+\RotRott\tM_{\ell}$
with $\tN_{\ell}\in\tH{n+1}{\bbS}(\om_{\ell})$ and $\tM_{\ell}\in\tH{n+2}{\bbS}(\om_{\ell})$.
In $\om_{\ell}$ we observe with Lemma \ref{cutlem}
\begin{align*}
\varphi_{\ell}\tN
&=\varphi_{\ell}\phi_{\ell}\tN
=\phi_{\ell}\tN_{\ell}
+\phi_{\ell}\RotRott\tM_{\ell}\\
&=\underbrace{\phi_{\ell}\tN_{\ell}
-\Psi(\na\na\phi_{\ell},\tM_{\ell})}_{=:\tN_{1,\ell}}
-\underbrace{2\sym\big((\spn\na\phi_{\ell})\Rot\tM_{\ell}\big)}_{=:\tN_{2,\ell}}
+\RotRott(\phi_{\ell}\tM_{\ell})\\
&\in\tH{n+1}{\bbS}(\om_{\ell})+\RotRott\tH{n+2}{\bbS}(\om_{\ell}).
\end{align*}
Extending $\tN_{1,\ell}+\tN_{2,\ell}$ and $\phi_{\ell}\tM_{\ell}$ by zero gives
$\widetilde\tN_{\ell}\in\tH{n+1}{\bbS}(\om)$ 
and $\widetilde\tM_{\ell}\in\tH{n+2}{\bbS}(\om)$ with 
$$\tN
=\sum_{\ell=1}^{L}\varphi_{\ell}\tN
=\sum_{\ell=1}^{L}\widetilde\tN_{\ell}+\RotRott\sum_{\ell=1}^{L}\widetilde\tM_{\ell}
\in\tH{n+1}{\bbS}(\om)+\RotRott\tH{n+2}{\bbS}(\om),$$
and all applied operations are continuous, showing the existence 
of $\Qot_{\Divs}^{n}$ and $\widetilde{\Qot}_{\RotRotst}^{n}$.

$\bullet$ Let $\tM\in\tHc{m,m-1}{\bbS}(\RotRott,\om)$ and $m\geq1$. 
Lemma \ref{cutlem}, approximation, and Corollary \ref{regdecononstandcor} show
$$\varphi_{\ell}\tM
\in\tHc{m,m-1}{\bbS}(\RotRott,\om_{\ell})
=\tHc{m+1}{\bbS}(\om_{\ell})+\symna\vHc{m+1}{}(\om_{\ell}).$$
Thus $\varphi_{\ell}\tM=\tM_{\ell}+\symna\vv_{\ell}$
with $\tM_{\ell}\in\tHc{m+1}{\bbS}(\om_{\ell})$ and $\vv_{\ell}\in\vHc{m+1}{}(\om_{\ell})$.
Extending $\tM_{\ell}$ and $\vv_{\ell}$ by zero gives
$\widetilde\tM_{\ell}\in\tHc{m+1}{\bbS}(\om)$ and $\widetilde\vv_{\ell}\in\vHc{m+1}{}(\om)$ with 
$$\tM
=\sum_{\ell=1}^{L}\varphi_{\ell}\tM
=\sum_{\ell=1}^{L}\widetilde\tM_{\ell}+\symna\sum_{\ell=1}^{L}\widetilde\vv_{\ell}
\in\tHc{m+1}{\bbS}(\om)+\symna\vHc{m+1}{}(\om),$$
and all applied operations are continuous, i.e.,
$\Qot_{\RotRotcst}^{m,m-1}$ and $\widetilde{\Qot}_{\symnac}^{m,m-1}$ exist.

$\bullet$ Let $\tM\in\tH{n,n-1}{\bbS}(\RotRott,\om)$.
By Lemma \ref{cutlem}, approximation, and Corollary \ref{regdecononstandcor} we get
$$\varphi_{\ell}\tM
\in\tH{n,n-1}{\bbS}(\RotRott,\om_{\ell})
=\tH{n+1}{\bbS}(\om_{\ell})+\symna\vH{n+1}{}(\om_{\ell})$$
and $\varphi_{\ell}\tM=\tM_{\ell}+\symna\vv_{\ell}$
with $\tM_{\ell}\in\tH{n+1}{\bbS}(\om_{\ell})$ and $\vv_{\ell}\in\vH{n+1}{}(\om_{\ell})$.
In $\om_{\ell}$ we observe
\begin{align*}
\varphi_{\ell}\tM
&=\varphi_{\ell}\phi_{\ell}\tM
=\phi_{\ell}\tM_{\ell}
+\phi_{\ell}\symna\vv_{\ell}\\
&=\underbrace{\phi_{\ell}\tM_{\ell}
-\sym(\na\phi_{\ell}\,\vv_{\ell}^{\top})}_{=:\tM_{1,\ell}}
+\symna(\phi_{\ell}\vv_{\ell})
\in\tH{n+1}{\bbS}(\om_{\ell})+\symna\vH{n+1}{}(\om_{\ell}).
\end{align*}
Extending $\tM_{1,\ell}$ and $\phi_{\ell}\vv_{\ell}$ by zero gives
$\widetilde\tM_{\ell}\in\tH{n+1}{\bbS}(\om)$ 
and $\widetilde\vv_{\ell}\in\vH{n+1}{}(\om)$ with 
$$\tM
=\sum_{\ell=1}^{L}\varphi_{\ell}\tM
=\sum_{\ell=1}^{L}\widetilde\tM_{\ell}+\symna\sum_{\ell=1}^{L}\widetilde\vv_{\ell}
\in\tH{n+1}{\bbS}(\om)+\symna\vH{n+1}{}(\om),$$
and all applied operations are continuous, i.e.,
$\Qot_{\RotRotst}^{n,n-1}$ and $\widetilde{\Qot}_{\symna}^{n,n-1}$ exist.
\end{proof}

\begin{rem}
\label{rem:regdecoformulaela1}
Let $m\in \n$ and $n\in \nz$. 
Since we have
$$\tHc{m}{\bbS}(\RotRott,\om)\subset\tHc{m,m-1}{\bbS}(\RotRott,\om),\quad
\tH{n}{\bbS}(\RotRott,\om)\subset\tH{n,n-1}{\bbS}(\RotRott,\om),$$
regular decompositions of $\tHc{m}{\bbS}(\RotRott,\om)$ 
and $\tH{n}{\bbS}(\RotRott,\om)$ follow:
\begin{align*}
\tHc{m}{\bbS}(\RotRott,\om)
&=\tHc{m+1,m}{\bbS}(\RotRott,\om)+\symna\vHc{m+1}{}(\om)
=\tHc{m+2}{\bbS}(\om)+\symna\vHc{m+1}{}(\om),\\
\tH{n}{\bbS}(\RotRott,\om)
&=\tH{n+1,n}{\bbS}(\RotRott,\om)+\symna\vH{n+1}{}(\om)
=\tH{n+2}{\bbS}(\om)+\symna\vH{n+1}{}(\om)
\end{align*}
hold by applying Theorem \ref{regdecopotgenlip} twice,
first on the level $(m,m-1)$ (resp.~$(n,n-1)$) and then on the level $(m+1,m)$ (resp.~$(n+1,n)$). 
\end{rem}

In summary, the results in Theorem \ref{regdecopotgenlip} and Remark \ref{rem:regdecoformulaela1} 
show that we have regular decompositions for all involved domains in the domain complexes 
\eqref{ElastComplfullVar1} and \eqref{ElastComplfullVar2} 
and $m,n\in\nz$ except for the space
$$\tHc{}{\bbS}(\RotRott,\om),$$
which we shall address in a forthcoming publication
even for partial boundary conditions.

Now we will finally show the following crucial compact embedding results.

\begin{theo}[Compact embeddings for the elasticity complex]
\label{cptembela1}
Let $m\in\nz$.
The embeddings
\begin{align*}
\text{\bf(i)}&
&
\vHc{m+1}{}(\om)
&\cpt\vHc{m}{}(\om),\\
\text{\bf(ii)}&
&
\tHc{m}{\bbS}(\RotRott,\om)\cap\tH{m}{\bbS}(\Div,\om)
&\cpt\tHc{m}{\bbS}(\om),\\
\text{\bf(iii)}&
&
\tHc{m}{\bbS}(\Div,\om)\cap\tH{m}{\bbS}(\RotRott,\om)
&\cpt\tHc{m}{\bbS}(\om),\\
\text{\bf(iv)}&
&
\vH{m+1}{}(\om)
&\cpt\vH{m}{}(\om)
\end{align*}
are compact and all ranges in the elasticity domain complexes are closed.
\end{theo}

\begin{proof}
(i) and (iv) are just Rellich's selection theorems.
With Theorem \ref{regdecopotgenlip} and Rellich's selection theorem
we can apply Theorem \ref{cptembmaintheo} 
(with $\Ao=\Divcs$)
and Corollary \ref{cptembmaintheodual}
(with $\Azs=\Divs$)
and get the compact embedding results (iii) and (ii)
for the case $m=0$.
For (ii) and $m\geq1$, let $(\tM_{k})\subset\tHc{m}{\bbS}(\RotRott,\om)\cap\tH{m}{\bbS}(\Div,\om)$
be bounded in $\tHc{m}{\bbS}(\RotRott,\om)\cap\tH{m}{\bbS}(\Div,\om)$.
By assumption (w.l.o.g.) $(\tM_{k})\subset\tHc{m}{\bbS}(\om)$
is a Cauchy sequence in $\tHc{m-1}{\bbS}(\om)$. 
Moreover, for all $|\alpha|=m$ we have
$\p^{\alpha}\tM_{k}\in\tH{}{\bbS}(\RotRott,\om)\cap\tH{}{\bbS}(\Div,\om)$. 
As $\tM_{k},\RotRott\tM_{k}\in\tHc{m}{\bbS}(\om)$ and $\tM_{k}\in D(\RotRotcst)$
we observe for all $\tPhi\in\tCc{\infty}{}(\reals^{3})$
\begin{align*}
\scp{\p^{\alpha}\tM_{k}}{\RotRott\tPhi}_{\tL{2}{\bbS}(\om)}
&=\pm\scp{\tM_{k}}{\RotRott\p^{\alpha}\tPhi}_{\tL{2}{\bbS}(\om)}\\
&=\pm\scp{\RotRott\tM_{k}}{\p^{\alpha}\tPhi}_{\tL{2}{\bbS}(\om)}
=\scp{\RotRott\p^{\alpha}\tM_{k}}{\tPhi}_{\tL{2}{\bbS}(\om)}.
\end{align*}
Thus $\p^{\alpha}\tM_{k}\in\tHc{}{\bbS}(\RotRott,\om)\cap\tH{}{\bbS}(\Div,\om)$
and hence w.l.o.g. $(\p^{\alpha}\tM_{k})$ is a Cauchy sequence in $\tL{2}{\bbS}(\om)$
by Theorem \ref{cptembela1} for $m=0$.
We conclude that $(\tM_{k})$ is a Cauchy sequence in $\tHc{m}{\bbS}(\om)$.
Analogously we prove (iii) for $m\geq1$.
\end{proof}

\begin{rem}[Compact embeddings for the elasticity complex]
\label{rem:cptembela1}
Theorem \ref{cptembela1} may also be proved by a variant of 
Theorem \ref{cptembmaintheo} (or Corollary \ref{cptembmaintheodual}).
For this, let, e.g., 
$$(\tM_{k})\subset\tHc{m}{\bbS}(\Div,\om)\cap\tH{m}{\bbS}(\RotRott,\om)$$
be a bounded sequence. In particular, $(\tM_{k})$ is bounded in $\tH{m,m-1}{\bbS}(\RotRott,\om)$.
According to Theorem \ref{regdecopotgenlip} we decompose
$\tM_{k}=\widetilde{\tM}_{k}+\symna\vv_{k}$
with $\widetilde{\tM}_{k}\in\tH{m+1}{\bbS}(\om)$
and $\vv_{k}\in\vH{m+1}{}(\om)$. By the boundedness of the potentials,
$(\widetilde{\tM}_{k})$ and $(\vv_{k})$ are bounded in $\tH{m+1}{\bbS}(\om)$ and $\vH{m+1}{}(\om)$
and hence w.l.o.g. converge in $\tH{m}{\bbS}(\om)$ and $\vH{m}{}(\om)$, respectively.
Let $\alpha$ with $|\alpha|\leq m$. Then
$$\p^{\alpha}\tM_{k}\in\tH{}{\bbS}(\Div,\om)\cap\tH{}{\bbS}(\RotRott,\om).$$
Moreover, as $\tM_{k}\in\tHc{m}{\bbS}(\om)\cap D(\Divcs)$
and $\Div\tM_{k}\in\vHc{m}{}(\om)$
we observe for all $\vphi\in\vCc{\infty}{}(\reals^{3})$
\begin{align*}
\scp{\p^{\alpha}\tM_{k}}{\symna\vphi}_{\tL{2}{\bbS}(\om)}
&=\pm\scp{\tM_{k}}{\symna\p^{\alpha}\tPhi}_{\tL{2}{\bbS}(\om)}\\
&=\mp\scp{\Div\tM_{k}}{\p^{\alpha}\vphi}_{\vL{2}{}(\om)}
=-\scp{\Div\p^{\alpha}\tM_{k}}{\tPhi}_{\vL{2}{}(\om)}.
\end{align*}
Thus $\p^{\alpha}\tM_{k}\in\tHc{}{\bbS}(\Div,\om)\cap\tH{}{\bbS}(\RotRott,\om)$.
With $\tM_{k,l}:=\tM_{k}-\tM_{l}$, 
$\widetilde{\tM}_{k,l}:=\widetilde{\tM}_{k}-\widetilde{\tM}_{l}$, and $\vv_{k,l}:=\vv_{k}-\vv_{l}$
we get
\begin{align*}
\norm{\tM_{k,l}}_{\tHc{m}{\bbS}(\om)}^2
&=\scp{\tM_{k,l}}{\widetilde{\tM}_{k,l}+\symna\vv_{k,l}}_{\tHc{m}{\bbS}(\om)}
=\scp{\tM_{k,l}}{\widetilde{\tM}_{k,l}}_{\tH{m}{\bbS}(\om)}
-\scp{\Div\tM_{k,l}}{\vv_{k,l}}_{\vH{m}{}(\om)}\\
&\leq c\big(\norm{\widetilde{\tM}_{k,l}}_{\tH{m}{\bbS}(\om)}+\norm{\vv_{k,l}}_{\vH{m}{}(\om)}\big)\to0,
\end{align*}
finishing the proof.
Alternatively, we can decompose $\tHc{m}{\bbS}(\Div,\om)$ using Theorem \ref{regdecopotgenlip},
which is even simpler.
\end{rem}

Remark \ref{rem:cptembela1} shows that we even have the following compact  embeddings:

\begin{theo}[More compact embeddings for the elasticity complex]
\label{cptembela2}
Let $m\in\nz$.
The embeddings
$$\tHc{m}{\bbS}(\Div,\om)\cap\tH{m,m-1}{\bbS}(\RotRott,\om)
\cpt\tHc{m}{\bbS}(\om)$$
are compact. In particular, for $m=0$
$$\tHc{}{\bbS}(\Div,\om)\cap\tH{0,-1}{\bbS}(\RotRott,\om)
\cpt\tL{2}{\bbS}(\om)$$
is compact. Similarly, for $m\geq1$ also the embeddings
$$\tHc{m,m-1}{\bbS}(\RotRott,\om)\cap\tH{m}{\bbS}(\Div,\om)
\cpt\tHc{m}{\bbS}(\om)$$
are compact.
\end{theo}

From Corollary \ref{regdecompregpot} we immediately obtain the existence of regular potentials 
also in the case of general strong Lipschitz domains.

\begin{cor}[Regular potentials]
\label{regdecopotgenlip2}
Let $m,n\in\nz$.
The regular potential representations in Corollary \ref{cor1} hold in the sense
\begin{align*}
\RotRott\tHc{m,m-1}{\bbS}(\RotRott,\om)
&=\RotRott\tHc{m+1}{\bbS}(\om),\quad m\geq1,\\
\RotRott\tHc{m}{\bbS}(\RotRott,\om)
=\RotRott\tHc{m+1,m}{\bbS}(\RotRott,\om)
&=\RotRott\tHc{m+2}{\bbS}(\om),\quad m\geq1,\\
\RotRott\tH{n,n-1}{\bbS}(\RotRott,\om)
&=\RotRott\tH{n+1}{\bbS}(\om),\\
\RotRott\tH{n}{\bbS}(\RotRott,\om)
=\RotRott\tH{n+1,n}{\bbS}(\RotRott,\om)
&=\RotRott\tH{n+2}{\bbS}(\om),\\
\vHc{m}{\bot_{\RM}}(\om)
=\Div\tHc{m}{\bbS}(\Div,\om)
&=\Div\tHc{m+1}{\bbS}(\om),\\
\vH{n}{}(\om)
=\Div\tH{n}{\bbS}(\Div,\om)
&=\Div\tH{n+1}{\bbS}(\om)
\end{align*}
with linear and bounded regular potential operators, cf.~Theorem \ref{regpottheotoptriv},
\begin{align*}
\Pot_{\RotRotcst}^{m,m-1}
:\RotRott\tHc{m,m-1}{\bbS}(\RotRott,\om)
&\To\tHc{m+1}{\bbS}(\om),\quad
m\geq1,&\\
\Pot_{\RotRotcst}^{m}
:\RotRott\tHc{m}{\bbS}(\RotRott,\om)
&\To\tHc{m+2}{\bbS}(\om),\quad
m\geq1,&\\
\Pot_{\RotRotst}^{n,n-1}
:\RotRott\tH{n,n-1}{\bbS}(\RotRott,\om)
&\To\tH{n+1}{\bbS}(\om),\\
\Pot_{\RotRotst}^{n}
:\RotRott\tH{n}{\bbS}(\RotRott,\om)
&\To\tH{n+2}{\bbS}(\om),\\
\Pot_{\Divcs}^{m}
:\vHc{m}{\bot_{\RM}}(\om)
&\To\tHc{m+1}{\bbS}(\om),\\
\Pot_{\Divs}^{n}
:\vH{n}{}(\om)
&\To\tH{n+1}{\bbS}(\om).
\end{align*}
\end{cor}

Lemma \ref{lem:SpaceRegDeco} yields:

\begin{cor}[Direct regular decompositions]
\label{cor:regdecopotgenlip}
Let $m,n\in\nz$. 
Some of the regular decompositions in Theorem \ref{regdecopotgenlip} and Remark \ref{rem:regdecoformulaela1} are direct.
More precisely, 
\begin{align*}
\tHc{m,m-1}{\bbS}(\RotRott,\om)
&=R(\Pot_{\RotRotcst}^{m,m-1})
\dotplus\tHc{m}{\bbS,0}(\RotRott,\om),\quad
m\ge1,\\
\tHc{m}{\bbS}(\RotRott,\om)
&=R(\Pot_{\RotRotcst}^{m})
\dotplus\tHc{m}{\bbS,0}(\RotRott,\om),\quad
m\ge1,\\
\tH{n,n-1}{\bbS}(\RotRott,\om)
&=R(\Pot_{\RotRotst}^{n,n-1})
\dotplus\tH{n}{\bbS,0}(\RotRott,\om),\\
\tH{n}{\bbS}(\RotRott,\om)
&=R(\Pot_{\RotRotst}^{n})
\dotplus\tH{n}{\bbS,0}(\RotRott,\om),\\
\tHc{m}{\bbS}(\Div,\om)
&=R(\Pot_{\Divcs}^{m})
\dotplus\tHc{m}{\bbS,0}(\Div,\om),\\
\tH{n}{\bbS}(\Div,\om)
&=R(\Pot_{\Divs}^{n})
\dotplus\tH{n}{\bbS,0}(\Div,\om).
\end{align*}
\end{cor}

Theorem \ref{cptembela1} and Theorem \ref{cptembela2} show:

\begin{theo}[Compact elasticity Hilbert complexes]
\label{cptembela3}
Let $m,n\in\nz$.
The elasticity domain complexes
\begin{equation*}
\xymatrixcolsep{5ex}\xymatrixrowsep{1ex}
\xymatrix{
\{\boldsymbol{0}\} \ar[r]^-{\iota_{\{\boldsymbol{0}\}}} &
\vHc{m+1}{}(\om) \ar[r]^-{\symnac} &
\tHc{m}{\bbS}(\RotRott,\om) \ar[r]^-{\RotRotcst} &
\tHc{m}{\bbS}(\Div,\om) \ar[r]^-{\Divcs} &
\vHc{m}{}(\om) \ar[r]^-{\pi_{\RM}} &
\RM,\\
\{\boldsymbol{0}\} & \ar[l]_-{\pi_{\{\boldsymbol{0}\}}}
\vH{n}{}(\om) & \ar[l]_-{-\Divs}
\tH{n}{\bbS}(\Div,\om) & \ar[l]_-{\RotRotst}
\tH{n}{\bbS}(\RotRott,\om) & \ar[l]_-{-\symna}
\vH{n+1}{}(\om) & \ar[l]_-{\iota_{\RM}}
\RM
}
\end{equation*}
and, for $m\ge 1$
\begin{equation*}
\xymatrixcolsep{3ex}\xymatrixrowsep{1ex}
\xymatrix{
\{\boldsymbol{0}\} \ar[r]^-{\iota_{\{\boldsymbol{0}\}}} &
\vHc{m+1}{}(\om) \ar[r]^-{\symnac} &
\tHc{m,m-1}{\bbS}(\RotRott,\om) \ar[r]^-{\RotRotcst} &
\tHc{m-1}{\bbS}(\Div,\om) \ar[r]^-{\Divcs} &
\vHc{m-1}{}(\om) \ar[r]^-{\pi_{\RM}} &
\RM,\\
\{\boldsymbol{0}\} & \ar[l]_-{\pi_{\{\boldsymbol{0}\}}}
\vH{n-1}{}(\om) & \ar[l]_-{-\Divs}
\tH{n-1}{\bbS}(\Div,\om) & \ar[l]_-{\RotRotst}
\tH{n,n-1}{\bbS}(\RotRott,\om) & \ar[l]_-{-\symna}
\vH{n+1}{}(\om) & \ar[l]_-{\iota_{\RM}}
\RM
}
\end{equation*}
are compact Hilbert complexes, all ranges are \emph{closed}, and the linear and bounded operators 
from Theorem \ref{regdecopotgenlip} and Corollary \ref{regdecopotgenlip2}
are associated regular decomposition and potential operators, respectively.
If $\om$ is additionally topologically trivial, then the elasticity complexes are exact.
\end{theo}

Based on Theorem \ref{cptembela1}, the FA-ToolBox 
(Section \ref{sec:FAToolBox} and \cite{paulymaxconst2,P2019b,P2019a,paulyapostfirstordergen,paulyzulehner2019a}) 
immediately provides a long list of implications 
such as geometric Helmholtz type decompositions, regular decompositions, Friedrichs/Poincar\'e type estimates, 
closed ranges, see \cite{paulyzulehner2020aarxiv} for a comprehensive list of such results.

Finally, we want to highlight some of these important and useful implications.

\begin{theo}[Mini FA-ToolBox for the elasticity complex]
\label{minitbtheo1}
It holds:
\begin{itemize}
\item[\bf(i)]
All ranges of the elasticity complexes are closed.
\item[\bf(ii)]
The cohomology groups of generalised Dirichlet and Neumann tensors
$$\tHarm{}{\bbS,D}(\om)
=N(\RotRotcst)\cap N(\Divs),\quad
\tHarm{}{\bbS,N}(\om)
=N(\Divcs)\cap N(\RotRotst)$$
are finite-dimensional.
\item[\bf(iii)]
The geometric Helmholtz type decompositions 
\begin{align*}
\tL{2}{\bbS}(\om)
&=R(\symnac)\oplus_{\tL{2}{\bbS}(\om)}\tHarm{}{\bbS,D}(\om)\oplus_{\tL{2}{\bbS}(\om)}R(\RotRotst),\\
\tL{2}{\bbS}(\om)
&=R(\RotRotcst)\oplus_{\tL{2}{\bbS}(\om)}\tHarm{}{\bbS,N}(\om)\oplus_{\tL{2}{\bbS}(\om)}R(\symna)
\end{align*}
hold.
\item[\bf(iv)]
There exist (optimal) $c_{0},c_{1},c_{2}>0$ such that
the Friedrichs/Poincar\'e type estimates 
\begin{align*}
\forall\,\vv
&\in\vHc{1}{}(\om)
&
\norm{\vv}_{\vL{2}{}(\om)}
&\leq c_{0}\norm{\symna\vv}_{\tL{2}{\bbS}(\om)},\\
\forall\,\tN
&\in\tH{}{\bbS}(\Div,\om) \cap R(\symnac)
&
\norm{\tN}_{\tL{2}{\bbS}(\om)}
&\leq c_{0}\norm{\Div\tN}_{\vL{2}{}(\om)},\\
\forall\,\tM
&\in\tHc{}{\bbS}(\RotRott,\om)\cap R(\RotRotst)
&
\norm{\tM}_{\tL{2}{\bbS}(\om)}
&\leq c_{1}\norm{\RotRott\tM}_{\tL{2}{\bbS}(\om)},\\
\forall\,\tM
&\in\tH{}{\bbS}(\RotRott,\om)\cap R(\RotRotcst)
&
\norm{\tM}_{\tL{2}{\bbS}(\om)}
&\leq c_{1}\norm{\RotRott\tM}_{\tL{2}{\bbS}(\om)},\\
\forall\,\tN
&\in\tHc{}{\bbS}(\Div,\om) \cap R(\symna)
&
\norm{\tN}_{\tL{2}{\bbS}(\om)}
&\leq c_{2}\norm{\Div\tN}_{\vL{2}{}(\om)},\\
\forall\,\vv
&\in\vH{1}{\bot_{\RM}}(\om)
&
\norm{\vv}_{\vL{2}{}(\om)}
&\leq c_{2}\norm{\symna\vv}_{\tL{2}{\bbS}(\om)}
\end{align*}
hold.
\end{itemize}
\end{theo}

\begin{rem}[Dirichlet and Neumann tensors]
\label{re:dirneuela}
As shown in \cite{PW2020a},
the dimensions of the Dirichlet and Neumann tensor fields of the elasticity complexes are given by
multiples of the corresponding Dirichlet and Neumann vector fields 
of the de Rham complexes, respectively. More precisely,
$$\dim\tHarm{}{\bbS,D}(\om)=\dim\RM\cdot\dim\vHarm{}{D}(\om),\quad
\dim\tHarm{}{\bbS,N}(\om)=\dim\RM\cdot\dim\vHarm{}{N}(\om),$$
where $\dim\RM=6$,
$\dim\vHarm{}{D}(\om)+1$ is the number of connected components of $\reals^{3}\setminus\ol{\om}$
resp. $\ga$, and $\dim\vHarm{}{N}(\om)$ is the number of ``handles'' of $\om$, 
see \cite{picardboundaryelectro,PW2020a} for exact definitions.
See also \cite{ciarletciarletgeymonatkrasucki2007}
for $\dim\tHarm{}{\bbS,N}(\om)$.
\end{rem}

\bibliographystyle{plain} 
\bibliography{/Users/paule/GoogleDriveData/Tex/input/bibTex/paule-walter}

\appendix

\section{Some Formulas}
\label{app:Formulas}

\begin{lem}
\label{appformulasproof}
For smooth functions, vector fields, and tensor fields we have
\begin{itemize}
\item[\bf(i)]
$\Rot(u\,\tI)=-\spn\na u$,
\item[\bf(ii)]
$\tr\Rot\tM=2\div\spn^{-1}\skw\tM$,
\item[\bf(ii')]
$\tr\Rot\tM=0$, if $\tM$ is symmetric,
\item[\bf(iii)]
$\Div\spn\vv=-\rot \vv$,
\item[\bf(iii')]
$\Div\skw\tM=-\rot \vv$ for $\vv=\spn^{-1}\skw\tM$,
\item[\bf(iv)]
$\Rot\spn\vv=(\div\vv)\,\tI-(\na\vv)^{\top}$,
\item[\bf(iv')]
$\Rot\skw\tM=(\div\vv)\,\tI-(\na\vv)^{\top}$ 
for $\vv=\spn^{-1}\skw\tM$,
\item[\bf(v)]
$\skw\Rot\tM=\spn\vv$
with $2\vv=\Div\tM^{\top}-\na\tr\tM$,
\item[\bf(v')]
$2\Div\sym\Rot\tM=-2\Div\skw\Rot\tM=\rot\Div\tM^{\top}$,
\item[\bf(vi)]
$\skw\RotRott \tM = \RotRott \skw\tM$,
\item[\bf(vi')]
$\sym\RotRott \tM = \RotRott \sym\tM$.
\end{itemize}
These formulas hold for distributions as well.
\end{lem}

\begin{proof}
For (i)--(v') see \cite[Lemma A.1]{paulyzulehner2019a}. 
We show (vi), (vi'). 
For $\tM$ and $\vv=\spn^{-1}\skw\tM$
it follows from (iv) and (i) that
$$\Rot(\Rot\skw\tM)^{\top}
=\Rot\big((\div\vv)\tI\big)
=-\spn\na\div\vv.$$
Conversely, (v) and (ii) imply
$$2\skw\Rot(\Rot \tM)^\top
=\spn(\Div\Rot\tM-\na\tr\Rot\tM)
=-2\spn\na\div\vv,$$
showing (vi). (vi') easily follows by
\begin{align*}
\sym\RotRott\tM 
&=\RotRott\tM-\skw\RotRott\tM\\
&=\RotRott\tM-\RotRott\skw\tM
=\RotRott\sym\tM,
\end{align*}
completing the proof.
\end{proof}

\vspace*{5mm}\hrule\vspace*{3mm}

\end{document}